\documentclass[reqno]{amsart}
\usepackage{}
\usepackage{stmaryrd}
\usepackage{mathrsfs}
\usepackage{cases}
\usepackage{amsfonts}
\usepackage{amssymb}
\usepackage{amsmath}
\usepackage{tikz}

\newtheorem{theorem}{Theorem}[section]
\newtheorem{lemma}[theorem]{Lemma}
\newtheorem{proposition}[theorem]{Proposition}
\newtheorem{corollary}[theorem]{Corollary}
\theoremstyle{definition}

\newtheorem{remark}[theorem]{Remark}
\numberwithin{equation}{section}

\newcommand{\blankbox}[2]

\begin{document}
\title[embedding relations between $\alpha$-modulation spaces]{Full
characterization of embedding relations between $\alpha$-modulation spaces }
\author{WEICHAO GUO}
\address{School of Mathematical Sciences, Xiamen University, Xiamen, 361005,
P.R.China}
\email{weichaoguomath@gmail.com}
\author{DASHAN FAN}
\address{Department of Mathematics, University of Wisconsin-Milwaukee,
Milwaukee, WI 53201, USA}
\email{fan@uwm.edu}
\author{HUOXIONG WU}
\address{School of Mathematical Sciences, Xiamen University, Xiamen, 361005,
P.R.China}
\email{huoxwu@xmu.edu.cn}
\author{GUOPING ZHAO}
\address{School of Applied Mathematics, Xiamen University of Technology,
Xiamen, 361024, P.R.China}
\email{guopingzhaomath@gmail.com}

\begin{abstract}
In this paper, we consider the embedding relations between any two $\alpha$%
-modulation spaces. Based on an observation that the $\alpha$-modulation
space with smaller $\alpha$ can be regarded as a corresponding $\alpha$%
-modulation space with larger $\alpha$, we give a complete characterization
of the Fourier multipliers between $\alpha$-modulation spaces with different
$\alpha$. Then we establish a full version of optimal embedding relations
between $\alpha$-modulation spaces.
\end{abstract}

\maketitle



\section{INTRODUCTION}

As we know, the decomposition method on frequency plays an important role in
the study of function spaces and their applications. Among many others, two
basic types of decomposition are used most frequently. One is the uniform
decomposition and the other is the dyadic decomposition. The corresponding
function spaces associated with these two decompositions are the modulation
spaces and the Besov spaces, respectively.

The modulation spaces, introduced by Feichtinger \cite{Feichtinger} in 1983,
was firstly defined by the short-time Fourier transform. We refer the reader
to see \cite{Feichtinger_Survey,Feichtinger,Grochenig,Wang_book} for some basic properties and
applications of the modulation spaces, as well as their historical
developments. Particularly, there is an equivalent definition of modulation
space using the uniform decomposition on the frequency plane. On the other
hand, it is well known that the Besov space $B_{p,q}^{s}$ (see \cite%
{Tribel_83}), constructed by the dyadic decomposition on frequency plane,
 is also a popular working frame in the fields of harmonic
analysis and partial differential equations.

In the eighties of last century, a so-called $\alpha$-covering on the
frequency plane was found in \cite{Feichtinger_I,Feichtinger_II}. This
covering is an intermediate decomposition method between the uniform
decomposition and the dyadic decomposition. Applying the $\alpha$-covering
to the frequency plane, Gr\"{o}bner \cite{Grobner_introduction} introduced
the $\alpha $-modulation spaces $M_{p,q}^{s,\alpha }$ with respect to the
parameters $\alpha \in \lbrack 0,1]$. The space $M_{p,q}^{s,\alpha }$
coincides with the modulation space $M_{p,q}^{s}$ when $\alpha=0$, and, in
some sense, it coincides with the Besov space $\ B_{p,q}^{s}$ when $\alpha
=1 $ (see \cite{Grobner_introduction}). So, for the sake of convenience, we
can view the Besov space as a special $\alpha$-modulation space and use $%
M_{p,q}^{s,1}$ to denote the inhomogeneous Besov space $B_{p,q}^{s}$.\bigskip

In the last ten years, the $\alpha $-modulation space received extensive
attention. Its many algebraic properties and geometric characterizations
were discovered, and many of its applications were established. For the
details, the reader can see \cite{Borup_Nielsen, Forasier} for the Banach
frames of $\alpha $-modulation spaces, \cite{Borup,Borup_Nielsen_operator,Zhao_Guo_Nolinear,Zhao_Guo_MathN}
for the boundedness of certain operators in the frame of
$\alpha $-modulation spaces.
We also refer the reader to \cite{Guo_Zhao_interpolation, Wang_Han}
for the study of some fundamental properties about $\alpha$-modulation spaces.
Among many features of the $\alpha $-modulation spaces,
a particularly interesting subject is the embedding between two different $%
\alpha $-modulation spaces. As an analogy of the Sobolev embedding on the
Lebesgue spaces, the embedding among the different $\alpha $-modulation
spaces plays a notable role in the study of partial differential equations
and in theory of function spaces. The research for embedding relation
between $\alpha _{1}$-modulation and $\alpha _{2}$-modulation spaces goes
back to Gr\"{o}bner's thesis \cite{Grobner_introduction}, where he
considered the case $1\leq p,q\leq \infty $. More sharp results were
established in \cite{Toft_embedding}, in which Toft and Wahlberg obtained
some partial sufficient conditions, as well as some partial necessary
conditions for such embedding between the $\alpha $-modulation spaces (One
also can see \cite{Sugimoto_Tomita_inlusion, Toft_Continunity, WZG_JFA_2006} for
the embedding between modulation and Besov spaces, and \cite%
{Kobayashi_inclusion} for the embedding between modulation spaces
and Sobolev spaces).
Especially, the embedding relations between $M_{p,q}^{s_{1},\alpha _{1}}$ and $M_{p,q}^{s_{2},\alpha _{2}}$
has been completely determined by Wang and Han in \cite{Wang_Han}. Their result can
be stated in the following proposition.

\begin{proposition}[\textbf{\protect\cite{Wang_Han}, Theorem 4.1 and Theorem 4.2}]
\label{old embedding} Let $0<p,q\leq \infty$, $s_i\in \mathbb{R}$, $\alpha_i
\in [0,1]$ for $i=1,2$. Then
\begin{equation}
M_{p,q}^{s_1,\alpha_1} \subset M_{p,q}^{s_2,\alpha_2}
\end{equation}
holds if and only if
\begin{equation*}
s_2+0\vee[n(\alpha_2-\alpha_1)(1/p-1/q)]\vee[n(\alpha_2-\alpha_1)(1-1/p-1/q)]%
\leq s_1,
\end{equation*}
where the notation \ $a\vee b$ denotes the maximum between $a$ and $b$.
\end{proposition}

We notice that all the previous results about the embedding relation between
$M_{p_{1},q_{1}}^{s_{1},\alpha _{1}}$ and $M_{p_{2},q_{2}}^{s_{2},%
\alpha_{2}} $ concern only some special $p_{1},p_{2},q_{1},q_{2}$. Hence, it
will be of great interest if we establish the sharp embedding $%
M_{p_{1},q_{1}}^{s_{1},\alpha _{1}}\subseteq
M_{p_{2},q_{2}}^{s_{2},\alpha_{2}}$ in full ranges $0<p_{i},q_{i}\leq \infty
,$ $s_{i}\in \mathbb{R},$ and $\alpha _{i}\in \lbrack 0,1]$ for $i=1,2$.
However, the complexity of the methods used in previous works make us quite
difficult to adopt these methods to treat more general situations. A
different and more efficient method might be necessarily introduced. Based
upon these motivation and observation, the main goal of this paper is to
seek a new method to give a complete characterization of the embedding
relation between any two $\alpha$-modulation spaces. The following theorem
is our main result.

\begin{theorem}[\textbf{Sharpness of embedding}]
\label{sharpness of embedding} Let $0<p_{i},q_{i}\leq \infty ,s_{i}\in
\mathbb{R},$ $\alpha _{i}\in \lbrack 0,1]$ \ for $i=1,2$. Then
\begin{equation}
M_{p_{1},q_{1}}^{s_{1},\alpha _{1}}\subseteq M_{p_{2},q_{2}}^{s_{2},\alpha
_{2}}
\end{equation}%
if and only if
\begin{equation}
\begin{cases}
\frac{1}{p_{2}}\leq \frac{1}{p_{1}} \\
s_{2}+R(\mathbf{p},\mathbf{q},\alpha _{1},\alpha _{2})\leq s_{1} \\
\frac{1}{q_{2}}\leq \frac{1}{q_{1}}%
\end{cases}%
,  \label{embedding condition 1}
\end{equation}%
or
\begin{equation}
\begin{cases}
\frac{1}{p_{2}}\leq \frac{1}{p_{1}} \\
s_{2}+R(\mathbf{p},\mathbf{q},\alpha _{1},\alpha _{2})+\frac{n(1-\alpha
_{1}\vee \alpha _{2})}{q_{2}}<s_{1}+\frac{n(1-\alpha _{1}\vee \alpha _{2})}{%
q_{1}} \\
\frac{1}{q_{2}}>\frac{1}{q_{1}}%
\end{cases}
\label{embedding condition 2}
\end{equation}%
holds. Here, we denote
\smaller{
\begin{equation*}\label{notation-R}
R (\mathbf{p},\mathbf{q};\alpha_1,\alpha_2)=
\begin{cases}
 \begin{aligned}
 \left[n\alpha_1(\frac{1}{p_1}-\frac{1}{p_2})\right]\vee
\left[n\alpha_2(1-\frac{1}{p_2})-n\alpha_1(1-\frac{1}{p_1})-n(\alpha_2-\alpha_1)\frac{1}{q_1}\right]
\\
\vee
\left[n(\alpha_2-\alpha_1)(\frac{1}{p_2}-\frac{1}{q_1})+n\alpha_1(\frac{1}{p_1}-\frac{1}{p_2})\right],
~\text{if}~\alpha_1\leq \alpha_2,
\\
 \left[n\alpha_2(\frac{1}{p_1}-\frac{1}{p_2})\right]\vee
\left[n\alpha_2(1-\frac{1}{p_2})-n\alpha_1(1-\frac{1}{p_1})-n(\alpha_2-\alpha_1)\frac{1}{q_2}\right]
\\
\vee
\left[n(\alpha_1-\alpha_2)(\frac{1}{q_2}-\frac{1}{p_1})+n\alpha_2(\frac{1}{p_1}-\frac{1}{p_2})\right],
~\text{if}~\alpha_1> \alpha_2,
\end{aligned}
\end{cases}
\end{equation*}} where $\mathbf{p}=(p_{1},p_{2})$, $\mathbf{q}=(q_{1},q_{2})$%
.
\end{theorem}

\begin{remark} Obviously, our result gives a complete characterization of the embedding
relations between any two $\alpha$-modulation spaces, which is an essential improvement and extension to Theorem A. We also remark that a similar embedding problem is studied
by Voigtlaender in \cite{decomposition space}, in which a more general framework, the decomposition spaces, is considered. However, our work has its own interesting.
\end{remark}

As we mentioned previously, neither the method in \cite{Toft_embedding}, nor
the proof used in proving Proposition \ref{old embedding} seems adoptable in
our proof on this more general situation. Thus, we will use a new and more
efficient approach to achieve our target. Below, we outline the strategy of
our proof for Theorem \ref{sharpness of embedding}.

Firstly, in Section 3, we will give a characterization of Fourier
multipliers between $\alpha $-modulation spaces by means of the
corresponding Wiener amalgam spaces. To be precise, in Theorem \ref%
{characterization of fourier multiplier}, we show that a Fourier multiplier $%
T_{m}$ is bounded from $M_{p_{1},q_{1}}^{s_{1},\alpha _{1}}$ to $\
M_{p_{2},q_{2}}^{s_{2},\alpha _{2}}$ if and only if the norm sequence
\begin{equation*}
\left\{ \Im _{k}\right\} =\left\{ \Vert \Box _{k}^{\alpha _{1}\vee \alpha
_{2}}T_{m}\Vert _{M_{p_{1},q_{1}}^{0,\alpha _{1}}\rightarrow
M_{p_{2},q_{2}}^{0,\alpha _{2}}}\right\} _{k\in \mathbb{Z}^{n}}
\end{equation*}%
is a pointwise multiplier from the sequence space $l_{q_{1}}^{s_{1},\alpha
_{1}\vee \alpha _{2}}$ to the sequence space $l_{q_{2}}^{s_{2},\alpha
_{1}\vee \alpha _{2}}$, where $\left\{ \Box _{k}^{\alpha _{1}\vee \alpha
_{2}}T_{m}\right\} $ is the sequence of localizations of $T_{m}$ based on
the $\alpha $-covering on the frequency plane. Since the embedding $%
M_{p_{1},q_{1}}^{s_{1},\alpha _{1}}\subset $\ $M_{p_{2},q_{2}}^{s_{2},\alpha
_{2}}$ can be viewed as the boundedness of the identity operator mapping
from $M_{p_{1},q_{1}}^{s_{1},\alpha _{1}}\ \ $to $\
M_{p_{2},q_{2}}^{s_{2},\alpha _{2}},$ it is reduced from Theorem \ref%
{characterization of fourier multiplier} that the embedding relation
$M_{p_{1},q_{1}}^{s_{1},\alpha _{1}}\subset M_{p_{2},q_{2}}^{s_{2},\alpha_{2}}$
holds if and only if the norm sequence
\begin{equation*}
\left\{ \mathfrak{R}_{k}\right\} _{k\in \mathbb{Z}^{n}}=\left\{ \Vert \Box
_{k}^{\alpha _{1}\vee \alpha _{2}}\Vert _{M_{p_{1},q_{1}}^{0,\alpha
_{1}}\rightarrow M_{p_{2},q_{2}}^{0,\alpha _{2}}}\right\} _{k\in \mathbb{Z}%
^{n}}
\end{equation*}%
is a pointwise multiplier from the sequence space $l_{q_{1}}^{s_{1},\alpha
_{1}\vee \alpha _{2}}\ $to the sequence space $l_{q_{2}}^{s_{2},\alpha
_{1}\vee \alpha _{2}}.$ Thus, to prove Theorem \ref{sharpness of embedding},
it suffices to find the precise asymptotic estimate of the sequence $\left\{
\mathfrak{R}_{k}\right\}_{k\in \mathbb{Z}^n} $. This task is quite technical, and it will be
finished in Lemma \ref{asymptotic estimates} for the case $q_{1}=q_{2}=q$ with
the help of complex interpolations and a constructive proof. The asymptotic
estimate for the general case can be finally obtained by invoking Lemma 4.1
and some technical treatments.
By the asymptotic estimates of local operators between $\alpha $-modulation spaces,
we can verify the sufficient and necessary conditions simultaneously, which completes the proof of Theorem \ref{sharpness of embedding}.

We explain the organization of this paper. In Section 2, we give some
definitions of function spaces treated in this paper. We also collect some
basic properties used in our proof. We establish the framework for the proof
of Theorem \ref{sharpness of embedding} in Section 3. In Section 3, we first
use Proposition \ref{proposition, viewpoint} to illustrate our viewpoint
about $\alpha $-modulation spaces with different $\alpha $. By the spirit of
this viewpoint, we give a complete characterization of Fourier multipliers
between any two $\alpha $-modulation spaces. Then we obtain the reduction of
Theorem \ref{sharpness of embedding}. In Section 4, we establish some
asymptotic estimates of local operators between $\alpha $-modulation spaces,
which is the quantity part for the proof of Theorem \ref{sharpness of
embedding}. Based on the preparatory work in Section 3 and Section 4,
we will complete the proof of our main theorem in Section 5. Also in Section 5,
we will make some comments about this paper.

\section{PRELIMINARIES}

We recall some notations. Let $C$ be a positive constant that may depend on $%
n,p_i,q_i,s_i,\alpha.$ The notation $X\lesssim Y$ denotes the statement that
$X\leq CY$, the notation $X\sim Y$ means the statement $X\lesssim Y \lesssim
X$, and the notation $X\simeq Y$ denotes the statement $X=CY$. For a
multi-index $k=(k_1,k_2,...k_n)\in \mathbb{Z}^{n}$, we denote $|k|_{\infty}:
=\max_{i=1,2...n}|k_i|$, and $\langle k\rangle: =(1+|k|^{2})^{\frac{1%
}{2}}.$

Let $\mathscr {S}:= \mathscr {S}(\mathbb{R}^{n})$ be the Schwartz space and $%
\mathscr {S}^{\prime }:=\mathscr {S}^{\prime }(\mathbb{R}^{n})$ be the space
of tempered distributions. We define the Fourier transform $\mathscr {F}f$
and the inverse Fourier transform $\mathscr {F}^{-1}f$ of $f\in \mathscr {S}(%
\mathbb{R}^{n})$ by
\begin{equation*}
\mathscr {F}f(\xi)=\hat{f}(\xi)=\int_{\mathbb{R}^{n}}f(x)e^{-2\pi ix\cdot
\xi}dx ~~ , ~~ \mathscr {F}^{-1}f(x)=\hat{f}(-x)=\int_{\mathbb{R}%
^{n}}f(\xi)e^{2\pi ix\cdot \xi}d\xi.
\end{equation*}

We recall some definitions of the function spaces treated in this paper. For
the convenience of doing calculation, we give the definition of $\alpha $%
-modulation spaces based on the decomposition method. This definition is
equivalent to the one in \cite{Grobner_introduction}. Suppose that $c>0$ and
$C>0$ are two appropriate constants. Choose a sequence of Schwartz functions
$\{\eta _{k}^{\alpha }\}_{k\in \mathbb{Z}^{n}}$ satisfying
\begin{equation}
\begin{cases}
|\eta _{k}^{\alpha }(\xi )|\geq 1,~\text{if}~|\xi -\langle k\rangle ^{\frac{%
\alpha }{1-\alpha }}k|<c\langle k\rangle ^{\frac{\alpha }{1-\alpha }}; \\
\mathbf{supp}\eta _{k}^{\alpha }\subset \{\xi :|\xi -\langle k\rangle ^{%
\frac{\alpha }{1-\alpha }}k|<C\langle k\rangle ^{\frac{\alpha }{1-\alpha }%
}\}; \\
\sum_{k\in \mathbb{Z}^{n}}\eta _{k}^{\alpha }(\xi )\equiv 1,\forall \xi \in
\mathbb{R}^{n}; \\
|\partial ^{\gamma }\eta _{k}^{\alpha }(\xi )|\leq C_{\alpha}\langle
k\rangle ^{-\frac{\alpha |\gamma |}{1-\alpha }},\forall \xi \in \mathbb{R}%
^{n},\gamma \in (\mathbb{Z}^{+}\cup \{0\})^{n}.%
\end{cases}%
\end{equation}%
Then $\{\eta _{k}^{\alpha }(\xi )\}_{k\in \mathbb{Z}^{n}}$ constitutes a
smooth decomposition of $\mathbb{R}^{n}$. The frequency decomposition
operators associated with the above function sequence can be defined by
\begin{equation}
\Box _{k}^{\alpha }:=\mathscr {F}^{-1}\eta _{k}^{\alpha }\mathscr {F}
\end{equation}%
for $k\in \mathbb{Z}^{n}$. Let $0<p,q\leq \infty $, $s\in \mathbb{R}$, $%
\alpha \in \lbrack 0,1)$. The $\alpha $-modulation space associated with the
above decomposition is defined by
\begin{equation}
M_{p,q}^{s,\alpha }(\mathbb{R}^{n})=\{f\in \mathscr {S}^{\prime }(\mathbb{R}%
^{n}):\Vert f\Vert _{M_{p,q}^{s,\alpha }(\mathbb{R}^{n})}=\left( \sum_{k\in
\mathbb{Z}^{n}}\langle k\rangle ^{\frac{sq}{1-\alpha }}\Vert \Box
_{k}^{\alpha }f\Vert _{p}^{q}\right) ^{\frac{1}{q}}<\infty \}
\end{equation}%
with the usual modifications when $q=\infty $. For simplicity, we denote $%
M_{p,q}^{s}=M_{p,q}^{s,0}$ and $\eta _{k}(\xi )=\eta _{k}^{0}(\xi )$.

\begin{remark}
We recall that the above definition is independent of the choice of exact $%
\eta _{k}^{\alpha }$ (see \cite{Wang_Han}). Also, for sufficiently small $%
\delta >0$, one can construct a function sequence $\{\eta _{k}^{\alpha }(\xi
)\}_{k\in \mathbb{Z}^{n}}$ such that $\eta _{k}^{\alpha }(\xi )=1$ and $\eta
_{k}^{\alpha }(\xi )\eta _{l}^{\alpha }(\xi )=0$ if $k\neq l$, when $\xi $
lies in the ball $B(\langle k\rangle ^{\frac{1-\alpha }{\alpha }}k,\langle
k\rangle ^{\frac{1-\alpha }{\alpha }}\delta )$ (see \cite{Borup_Nielsen,
Forasier, Guo_Chen}).
\end{remark}

For $\alpha\in [0,1)$, we set
\begin{equation}
\Lambda_k^{\alpha}=\{l\in \mathbb{Z}^n:~\Box_l^{\alpha}\circ
\Box_k^{\alpha}\neq 0\}
\end{equation}
and
\begin{equation}
\Lambda_k^{\alpha,\ast}=\{l\in \mathbb{Z}^n:~\Box_l^{\alpha}\circ
\Box_m^{\alpha}\neq 0~\text{for}~\text{some}~m\in \Lambda_k^{\alpha}\}.
\end{equation}
Denote
\begin{equation}
\Box_k^{\alpha,\ast}=\sum_{l\in \Lambda_k^{\alpha}}\Box_l^{\alpha}, \hspace{%
6mm} \eta_k^{\alpha,\ast}=\sum_{l\in \Lambda_k^{\alpha}}\eta_l^{\alpha}.
\end{equation}

For $\alpha_1, \alpha_2 \in (0,1)$, $k\in \mathbb{Z}^n$, we denote
\begin{equation}
\Gamma_k^{\alpha_1,\alpha_2}=\{l\in \mathbb{Z}^n:~\Box_k^{\alpha_2}\circ
\Box_l^{\alpha_1}\neq 0 \},~ \widetilde{\Gamma_k^{\alpha_1,\alpha_2}}=\{l\in
\mathbb{Z}^n:~\Box_k^{\alpha_2}\circ \Box_l^{\alpha_1}=\Box_l^{\alpha_1} \}.
\end{equation}

To define the Besov space, we introduce the dyadic decomposition of $\mathbb{%
R}^{n}$. Let $\varphi(\xi)$ be a smooth bump function supported in the ball $%
\{\xi: |\xi|<\frac{3}{2}\}$ and be equal to 1 on the ball $\{\xi: |\xi|\leq
\frac{4}{3}\}$. 
For integers $j\in \mathbb{Z}$, we define the Littlewood-Paley operators
\begin{equation}
\begin{split}
\widehat{\Delta_jf}&=\left(\varphi(2^{-j}\xi)-\varphi(2^{-j+1}\xi)\right)%
\widehat{f}(\xi),\ j\geq 0, \\
\widehat{\Delta_0f}&=\varphi(\xi)f(\xi).
\end{split}%
\end{equation}
Let $0< p,q\leq\infty$ and $s\in \mathbb{R}$. For $f\in\mathscr {S}^{\prime
} $, set
\begin{equation}
\|f\|_{B_{p,q}^s}=\left(\sum_{j=0}^{\infty}2^{jsq}\|\triangle_jf\|_{L^p}^q
\right)^{1/q}.
\end{equation}
The (inhomogeneous) Besov space is the space of all tempered distributions $%
f $ for which the quantity $\|f\|_{B_{p,q}^s}$ is finite.

Suppose $0<q\leq \infty $, $s\in \mathbb{R}$ and $\alpha \in \lbrack 0,1)$.
Let $\{\lambda _{k}\}_{k\in \mathbb{Z}^{n}}$ denote a sequence of complex
numbers. Set
\begin{equation}
\Vert \{\lambda _{k}\}\Vert _{l_{q}^{s,\alpha }}=%
\begin{cases}
\left( \sum_{k\in \mathbb{Z}^{n}}^{{}}\langle k\rangle ^{\frac{sq}{1-\alpha }%
}|\lambda _{k}|^{q}\right) ^{\frac{1}{q}}~ & \text{if}~0<q<\infty , \\
\sup_{k\in \mathbb{Z}^{n}}\left( \langle k\rangle ^{\frac{s}{1-\alpha }%
}|\lambda _{k}|\right) ~ & \text{if}~q=\infty .%
\end{cases}%
\end{equation}%
We use $l_{q}^{s,\alpha }$ to denote the set of all sequences $\{\lambda
_{k}\}_{k\in \mathbb{Z}^{n}}$ such that $\Vert \{\lambda _{k}\}\Vert
_{l_{q}^{s,\alpha }}<\infty $.

Similarly, we use $l_{q}^{s,1}$ to denote the set of all sequences $%
\{\lambda _{j}\}_{j\in \mathbb{N}}$ such that
\begin{equation}
\Vert \{\lambda _{j}\}\Vert _{l_{q}^{s,1}}=%
\begin{cases}
\left( \sum_{j\in \mathbb{N}}2^{js}|a_{j}|^{q}\right) ^{\frac{1}{q}}~ &
\text{if}~0<q<\infty , \\
\sup_{j\in \mathbb{N}}\left( 2^{js}|a_{j}|\right) ~ & \text{if}~q=\infty
\end{cases}%
\end{equation}%
is finite.

Now, we define the space of pointwise multipliers between sequence spaces.
For $\alpha \in \lbrack 0,1)$, we set
\begin{equation}
\mathcal{M}_{p}(l_{q_{1}}^{s_{1},\alpha },l_{q_{2}}^{s_{2},\alpha })=\left\{
\{a_{k}\}_{k\in \mathbb{Z}^{n}}:~\Vert \{a_{k}\lambda _{k}\}\Vert
_{l_{q_{2}}^{s_{2},\alpha }}\lesssim \Vert \{\lambda _{k}\}\Vert
_{l_{q_{1}}^{s_{1},\alpha }}\text{ for all }\{\lambda _{k}\}\in
l_{q}^{s_1,\alpha }\right\} .
\end{equation}
For $\alpha =1$, we set
\begin{equation}
\mathcal{M}_{p}(l_{q_{1}}^{s_{1},1 },l_{q_{2}}^{s_{2},1 })=\left\{
\{a_{j}\}_{j\in \mathbb{N}}:~\Vert \{a_{j}\lambda _{j}\}\Vert
_{l_{q_{2}}^{s_{2},1 }}\lesssim \Vert \{\lambda _{j}\}\Vert
_{l_{q_{1}}^{s_{1},1 }}\text{ for all }\{\lambda _{j}\}\in
l_{q}^{s_1,1 }\right\} .
\end{equation}%
Denote
\begin{equation}
\Vert \{a_{k}\}|~\mathcal{M}_{p}(l_{q_{1}}^{s_{1},\alpha
},l_{q_{2}}^{s_{2},\alpha })\Vert =\Vert \{a_{k}\}\Vert
_{l_{q_{1}}^{s_{1},\alpha }\rightarrow l_{q_{2}}^{s_{2},\alpha
}}=\sup_{\Vert \{\lambda _{k}\}\Vert _{l_{q_{1}}^{s_{1},\alpha }=1}}\Vert
\{a_{k}\lambda _{k}\}\Vert _{l_{q_{2}}^{s_{2},\alpha }}.
\end{equation}

We list the following lemmas which will be used frequently in our proof.

\begin{lemma}[\textbf{Embedding of $L^{p}$ with
Fourier compact support, \cite{Tribel_83}}]
\label{embedding of Lp with Fourier compact support} Let $0<p_{1}\leq
p_{2}\leq \infty $ and assume $supp{\hat{f}}\subseteq B(0,R)$. We have
\begin{equation}
\Vert f\Vert _{L^{p_{2}}}\leq CR^{n(\frac{1}{p_{1}}-\frac{1}{p_{2}})}\Vert
f\Vert _{L^{p_{1}}},
\end{equation}%
where $C$ is independent of $f$.
\end{lemma}

\begin{lemma}[\textbf{Convolution in $L^{p}$ with $%
p<1$, \cite{Tribel_83}}]\label{lemma, convolution}
\label{convolution} Let $0<p<1$ and $L_{B(x_{0},R)}^{p}=\{f\in L^{p}(\mathbb{%
R}^{n}):\ \mathrm{supp}\widehat{f}\subset B(x_{0},R)\}$, where $%
B(x_{0},R)=\{x:|x-x_{0}|\leq R\}$. Suppose $f,g\in L_{B(x_{0},R)}^{p}$. Then
there exists a constant $C>0$ which is independent of $x_{0}$ and $R>0$ such
that
\begin{equation*}
\Vert f\ast g\Vert _{p}\leq CR^{n(1/p-1)}\Vert f\Vert _{p}\Vert g\Vert _{p}.
\end{equation*}
\end{lemma}

\begin{lemma}[\textbf{see \protect\cite{Wang_Han}}]
\label{positive result for complex interpolation} Let $0< p_i, q_i\leq
\infty $, $s_i\in \mathbb{R}$ for $i=1,2$ and $\alpha \in [0,1]$. Then we
have
\begin{equation}
[M_{p_1, q_1}^{s_1, \alpha}, M_{p_2, q_2}^{s_2,
\alpha}]_{\theta}=M_{p_{\theta}, q_{\theta}}^{s_{\theta}, \alpha}
\end{equation}
for $\theta\in (0,1)$, where
\begin{equation*}
\frac{1}{p_{\theta }}=\frac{1-\theta }{p_{1}}+\frac{\theta }{p_{2}},~\frac{1%
}{q_{\theta }}=\frac{1-\theta }{q_{1}}+\frac{\theta }{q_{1}},~s_{\theta
}=(1-\theta )s_{1}+\theta s_{2}.
\end{equation*}
\end{lemma}

\begin{remark}
In the rest of this paper, for simplicity in the notation, we denote
\begin{equation}
M_{i}^{s_{i}}=M_{p_{i},q_{i}}^{s_{i},\alpha _{i}},\hspace{6mm}
M_{i}=M_{p_{i},q_{i}}^{0,\alpha _{i}}
\end{equation}
for $i=1,2$, when no confusion is possible.
\end{remark}

\section{Fourier multipliers on $\protect\alpha$-modulation spaces}

In this section, we display some propositions to explain the framework for the
proof of Theorem \ref{sharpness of embedding}. And each of these
propositions also has its independent significance.

Firstly, we recall the previous study of Fourier multiplier on frequency decomposition spaces.
In Feichtinger-Narimani \cite{Feichtinger_Narimani}, the authors study the Fourier multiplier between $M_{p_1,q_1}$ and $M_{p_2,q_2}$,
where $1\leq p_i, q_i\leq \infty$ for $i=1,2$, one can also see Feichtinger-Gr\"{o}bner \cite{Feichtinger_I} for
a general result in the frame of Banach space (with same decomposition).
Recently, in order to study the behavior of unimodular multiplier on $\alpha$-modulation spaces, in \cite{Zhao_Guo_Nolinear},
we establish a corresponding result between $M_{p_1,q_1}^{s_1,\alpha}$ and $M_{p_2,q_2}^{s_2,\alpha}$, where
$1\leq p_i. q_i\leq \infty$, $s_i\in \mathbb{R}$.

In this section, we give a full characterization of Fourier multipliers between any two $\alpha $-modulation spaces,
which extends all the previous results.
Especially, our theorem
covers the case that $\alpha _{1}\neq \alpha _{2}$ and $s_1\neq s_2$, which allows the different decompositions and different potentials.
Our theorem also covers the Quasi-Banach case ($p<1$ or $q<1$), which is not contained in the previous results.

Our method is based on the observation that an $\alpha$-modulation space
with smaller $\alpha$ can be regarded as a corresponding $\alpha$-modulation
space with larger $\alpha$. The following proposition demonstrates this
viewpoint.

\begin{proposition}
\label{proposition, viewpoint} Let $0< p, q \leq \infty,$ $s\in \mathbb{R}$,
$\alpha_i\in [0,1]$ and $\alpha_1\leq\alpha_2$. We have
\begin{equation}
\|f\|_{M_{p,q}^{s,\alpha_1}} \sim
\begin{cases}
\big\|\{\|\Box_k^{\alpha_2}f\|_{M_{p,q}^{0,\alpha_1}}\}|~{l_{q}^{s,\alpha_2}}%
\big\|,~ & \text{if}~\alpha_2<1 \\
\big\|\{\|\Delta_jf\|_{M_{p,q}^{0,\alpha_1}}\}|~{l_{q}^{s,1}}\big\|,~ &
\text{if}~\alpha_2=1.%
\end{cases}%
\end{equation}
\end{proposition}

\begin{proof}
We only state the proof for $\alpha_{2}<1,\ q<\infty$,
since the proof for the other cases shares the same idea.
Firstly, by the definition we have
\begin{equation}
\|\Box_k^{\alpha_2}f\|_{M_{p,q}^{0,\alpha_1}}=\left(\sum_{l\in
\Gamma_k^{\alpha_1,\alpha_2}}\|\Box_l^{\alpha_1}\Box_k^{\alpha_2}f\|^q_{L^p}%
\right)^{1/q}.
\end{equation}
Using Lemma \ref{lemma, convolution} or the Young's inequality, we deduce
\begin{equation}
\|\Box_l^{\alpha_1}\Box_k^{\alpha_2}f\|_{L^p}
\lesssim
\langle k\rangle^{\frac{\alpha_2 n}{1-\alpha_2}(\frac{1}{p\wedge 1}-1)}
\|\mathscr{F}^{-1}\eta_k^{\alpha_2}\|_{L^{p\wedge 1}}\|\Box_l^{\alpha_1}f\|_{L^p}
\lesssim
\|\Box_l^{\alpha_1}f\|_{L^p},
\end{equation}
it follows that
\begin{equation}
\|\Box_k^{\alpha_2}f\|_{M_{p,q}^{0,\alpha_1}}
\lesssim \left(\sum_{l\in
\Gamma_k^{\alpha_1,\alpha_2}}\|\Box_l^{\alpha_1}f\|^q_{L^p}\right)^{1/q}.
\end{equation}

Observing $|\Gamma_l^{\alpha_2,\alpha_1}|\lesssim 1$ for $\alpha_1\leq \alpha_2$, we deduce
\begin{equation}
\begin{split}
\big\|\{\|\Box_k^{\alpha_2}f\|_{M_{p,q}^{0,\alpha_1}}\}|~{l_{q}^{s,\alpha_2}}%
\big\| \lesssim & \left(\sum_{k\in \mathbb{Z}^n}\langle k\rangle^{\frac{sq}{%
1-\alpha_2}}\sum_{l\in
\Gamma_k^{\alpha_1,\alpha_2}}\|\Box_l^{\alpha_1}f\|^q_{L^p}\right)^{1/q} \\
\sim & \left(\sum_{l\in \mathbb{Z}^n}\sum_{k\in
\Gamma_l^{\alpha_2,\alpha_1}}\langle l\rangle^{\frac{sq}{1-\alpha_1}%
}\|\Box_l^{\alpha_1}f\|^q_{L^p}\right)^{1/q} \\
\lesssim & \left(\sum_{l\in \mathbb{Z}^n}\langle l\rangle^{\frac{sq}{%
1-\alpha_1}}\|\Box_l^{\alpha_1}f\|^q_{L^p}\right)^{1/q}\lesssim
\|f\|_{M_{p,q}^{s,\alpha_1}}.
\end{split}%
\end{equation}
On the other hand, we have
\begin{equation}
\|\Box_l^{\alpha_1}f\|_{L^p}=\|\sum_{k\in
\Gamma_l^{\alpha_2,\alpha_1}}\Box_k^{\alpha_2}\Box_l^{\alpha_1}f\|_{L^p}.
\end{equation}
Observing $|\Gamma_l^{\alpha_2,\alpha_1}|\lesssim 1,$ we deduce
\begin{equation}
\|\sum_{k\in
\Gamma_l^{\alpha_2,\alpha_1}}\Box_k^{\alpha_2}\Box_l^{\alpha_1}f\|^q_{L^p}
\lesssim \sum_{k\in
\Gamma_l^{\alpha_2,\alpha_1}}\|\Box_k^{\alpha_2}\Box_l^{\alpha_1}f\|^q_{L^p}.
\end{equation}
This leads to
\begin{equation}  \label{for proof 1}
\begin{split}
\|f\|_{M_{p,q}^{s,\alpha_1}} =& \left(\sum_{l\in \mathbb{Z}^n}\langle
l\rangle^{\frac{sq}{1-\alpha_1}}\|\Box_l^{\alpha_1}f\|^q_{L^p}\right)^{1/q}
\\
\lesssim& \left(\sum_{l\in \mathbb{Z}^n}\langle l\rangle^{\frac{sq}{%
1-\alpha_1}}\sum_{k\in
\Gamma_l^{\alpha_2,\alpha_1}}\|\Box_k^{\alpha_2}\Box_l^{\alpha_1}f\|^q_{L^p}%
\right)^{1/q} \\
\sim& \left(\sum_{l\in \mathbb{Z}^n}\sum_{k\in
\Gamma_l^{\alpha_2,\alpha_1}}\langle k\rangle^{\frac{sq}{1-\alpha_2}%
}\|\Box_k^{\alpha_2}\Box_l^{\alpha_1}f\|^q_{L^p}\right)^{1/q}.
\end{split}%
\end{equation}
By exchanging the summation order, we deduce that
\begin{equation}  \label{for proof 4}
\begin{split}
\left(\sum_{l\in \mathbb{Z}^n}\sum_{k\in
\Gamma_l^{\alpha_2,\alpha_1}}\langle k\rangle^{\frac{sq}{1-\alpha_2}%
}\|\Box_k^{\alpha_2}\Box_l^{\alpha_1}f\|^q_{L^p}\right)^{1/q} =&
\left(\sum_{k\in \mathbb{Z}^n}\sum_{l\in
\Gamma_k^{\alpha_1,\alpha_2}}\langle k\rangle^{\frac{sq}{1-\alpha_2}%
}\|\Box_k^{\alpha_2}\Box_l^{\alpha_1}f\|^q_{L^p}\right)^{1/q} \\
=& \left(\sum_{k\in \mathbb{Z}^n}\langle k\rangle^{\frac{sq}{1-\alpha_2}%
}\sum_{l\in
\Gamma_k^{\alpha_1,\alpha_2}}\|\Box_l^{\alpha_1}\Box_k^{\alpha_2}f\|^q_{L^p}%
\right)^{1/q} \\
=& \big\|\{\|\Box_k^{\alpha_2}f\|_{M_{p,q}^{0,\alpha_1}}\}|~{%
l_{q}^{s,\alpha_2}}\big\|.
\end{split}
\end{equation}
Combining with (\ref{for proof 1}) and (\ref{for proof 4}), we obtain our
conclusion.
\end{proof}

By the spirit of the above proposition, we are able to give a full
characterization of Fourier multipliers between any two $\alpha $-modulation
spaces in the following Theorem 3.2. Our theorem extends the known result in
\cite{Feichtinger_Narimani}, where the authors only consider the special
case $\alpha _{1}=\alpha _{2}=0$.
We would like to remark that our elementary method allows us to handle more general cases in Theorem \ref{characterization of fourier multiplier},
even including the Quasi-Banach case.

A tempered distribution $m$ is called a Fourier multiplier from $%
M_{p_{1},q_{1}}^{s_{1},\alpha _{1}}$ to $M_{p_{2},q_{2}}^{s_{2},\alpha _{2}}$%
, if there exists a constant $C>0$ such that
\begin{equation*}
\Vert T_{m}(f)\Vert _{M_{p_{2},q_{2}}^{s_{2},\alpha _{2}}}\leq C\Vert f\Vert
_{M_{p_{1},q_{1}}^{s_{1},\alpha _{1}}}
\end{equation*}%
for all $f\ $\ in the Schwartz space $\mathscr {S}(\mathbb{R}^{n})$, where
\begin{equation*}
T_{m}f=m(D)f=\mathscr {F}^{-1}(m\mathscr {F}f)
\end{equation*}%
is the Fourier multiplier operator associated with $m$, and $m$ is called
the symbol or multiplier of $T_{m}$.
Let $\mathcal{M}_{\mathscr{F}}\left(X,Y\right) $ denote the set of all symbols such that the corresponding Fourier multipliers are bounded from $X$ to $Y$.
We set
\begin{equation}
\Vert m|~\mathcal{M}_{\mathscr{F}}\left( X,Y\right) \Vert =\Vert T_{m}\Vert
_{X\rightarrow Y}=\sup \{\Vert m(D)f\Vert _{Y}:f\in \mathscr {S}(\mathbb{R}%
^{n}),\Vert f\Vert _{X}=1\},
\end{equation}%
where $X$ and $Y$ denote certain $\alpha $-modulation spaces.

For the sake of convenience, we define some exact Wiener amalgam spaces. For
$m\in \mathscr {S}^{\prime }$, $0<p_i, q_i\leq \infty$, $s_i\in \mathbb{R}$, $\alpha_i\in [0,1]$ for $i=1,2$,
we denote

\begin{equation}
\|m|~W^{\alpha}\left(\mathcal{M}_{\mathscr{F}}(M_{p_1,q_1}^{0,\alpha_1},
M_{p_2,q_2}^{0,\alpha_2}), \mathcal{M}_{p}(l_{q_1}^{s_1},
l_{q_2}^{s_2})\right)\|=
\|\{\|\Box_k^{\alpha}T_m\|_{M_{p_1,q_1}^{0,\alpha_1}\rightarrow
M_{p_2,q_2}^{0,\alpha_2}}\}\|_{l_{q_1}^{s_1,\alpha}\rightarrow
l_{q_2}^{s_2,\alpha}}
\end{equation}
for $\alpha\in [0,1)$. Similarly, we denote

\begin{equation}
\|m|~W^{1}\left(\mathcal{M}_{\mathscr{F}}(M_{p_1,q_1}^{0,\alpha_1},
M_{p_2,q_2}^{0,\alpha_2}), \mathcal{M}_{p}(l_{q_1}^{s_1},
l_{q_2}^{s_2})\right)\|=
\|\{\|\Delta_jT_m\|_{M_{p_1,q_1}^{0,\alpha_1}\rightarrow
M_{p_2,q_2}^{0,\alpha_2}}\}\|_{l_{q_1}^{s_1,1}\rightarrow l_{q_2}^{s_2,1}}.
\end{equation}

\begin{theorem}[\textbf{Characterization of Fourier multipliers on $\protect%
\alpha$-modulation spaces}]
\label{characterization of fourier multiplier} Let $0< p_i, q_i \leq \infty,$
$s_i\in \mathbb{R}$, $\alpha_i\in [0,1]$ for $i=1,2$. Then we have
\begin{equation}
\mathcal{M}_{\mathscr{F}}\left(M_{p_1,q_1}^{s_1,\alpha_1},
M_{p_2,q_2}^{s_2,\alpha_2}\right) = W^{\alpha_1\vee \alpha_2}\left(\mathcal{M%
}_{\mathscr{F}}\left(M_{p_1,q_1}^{0,\alpha_1},
M_{p_2,q_2}^{0,\alpha_2}\right), \mathcal{M}_{p}(l_{q_1}^{s_1},
l_{q_2}^{s_2})\right).
\end{equation}
\end{theorem}

\begin{proof}
We only give the proof for $\alpha_1, \alpha_2<1$, the other cases can be
handled similarly. We divide this proof into two cases.
\\
\textbf{Case 1. $\alpha_1\leq \alpha_2$.}

Firstly, we assume $m\in W^{\alpha_2}\left(\mathcal{M}_{\mathscr{F}}(M_1,
M_2), \mathcal{M}_{p}(l_{q_1}^{s_1}, l_{q_2}^{s_2})\right)$. By the
definition, we have that $\Box_k^{\alpha_2}T_m\in \mathcal{L}(M_1, M_2)$ for
$k\in \mathbb{Z}^n$, and $\{\|\Box_k^{\alpha_2}T_m\|_{M_1\rightarrow
M_2}\}\in \mathcal{M}_{p}(l_{q_1}^{s_1,\alpha_2}, l_{q_2}^{s_2,\alpha_2})$.
For any $f\in \mathscr {S}$, we deduce
\begin{equation}
\begin{split}
\|\Box_k^{\alpha_2}T_mf\|_{L^{p_2}} \sim&\|\Box_k^{\alpha_2}T_m f\|_{M_2} \\
=&\|\Box_k^{\alpha_2}T_m \Box_k^{\alpha_2,\ast}f\|_{M_2} \\
\lesssim& \|\Box_k^{\alpha_2}T_m\|_{M_1\rightarrow
M_2}\|\Box_k^{\alpha_2,\ast}f\|_{M_1}.
\end{split}%
\end{equation}
It yields
\begin{equation*}
\begin{split}
\|T_mf\|_{M_2^{s_2}}=&\|\{\|\Box_k^{\alpha_2}T_mf%
\|_{L^{p_2}}\}\|_{l_{q_2}^{s_2,\alpha_2}} \\
\lesssim& \|\{\|\Box_k^{\alpha_2}T_m\|_{M_1\rightarrow
M_2}\|\Box_k^{\alpha_2,\ast}f\|_{M_1}\}\|_{l_{q_2}^{s_2,\alpha_2}} \\
\lesssim& \|\{\|\Box_k^{\alpha_2}T_m\|_{M_1\rightarrow
M_2}\}\|_{l_{q_1}^{s_1,\alpha_2}\rightarrow l_{q_2}^{s_2,\alpha_2}}
\|\{\|\Box_k^{\alpha_2,\ast}f\|_{M_1}\}\|_{l_{q_1}^{s_1,\alpha_2}} \\
=& \big\|m|~W^{\alpha_2}\left(\mathcal{M}_{\mathscr{F}}(M_1, M_2),
\mathcal{M}_{p}(l_{q_1}^{s_1}, l_{q_2}^{s_2})\right)\big\| %
\|\{\|\Box_k^{\alpha_2,\ast}f\|_{M_1}\}\|_{l_{q_1}^{s_1,\alpha_2}}.
\end{split}
\end{equation*}
Observing
\begin{equation}
\begin{split}
\big\|\{\|\Box_k^{\alpha_2,\ast}f\|_{M_1}\}\big\|_{l_{q_1}^{s_1,\alpha_2}}
\lesssim &
\big\|\{\sum_{l\in \Lambda_k^{\alpha_2,\ast}}\|\Box_l^{\alpha_2}f\|_{M_1}\}_{k\in \mathbb{Z}^n}\big\|_{l_{q_1}^{s_1,\alpha_2}}
\\
\lesssim &
\big\|\{\|\Box_k^{\alpha_2}f\|_{M_1}\}_{k\in \mathbb{Z}^n}\big\|_{l_{q_1}^{s_1,\alpha_2}}
\sim \|f\|_{M_1^{s_1}},
\end{split}
\end{equation}
we obtain
\begin{equation*}
  \|T_mf\|_{M_2^{s_2}}
  \lesssim
  \big\|m|~W^{\alpha_2}\left(\mathcal{M}_{\mathscr{F}}(M_1, M_2),
  \mathcal{M}_{p}(l_{q_1}^{s_1}, l_{q_2}^{s_2})\right)\big\|
  \|f\|_{M_1^{s_1}},
\end{equation*}
which further implies
\begin{equation}
\big\|m|~\mathcal{M}_{\mathscr{F}}\left(M_1^{s_1}, M_2^{s_2}\right)\big\| %
\lesssim \big\|m|~W^{\alpha_2}\left(\mathcal{M}_{\mathscr{F}}(M_1, M_2),
\mathcal{M}_{p}(l_{q_1}^{s_1}, l_{q_2}^{s_2})\right)\big\|.
\end{equation}

Next, we assume $m\in \mathcal{M}_{\mathscr{F}}\left(M_1^{s_1},M_2^{s_2}\right)$.
For the local operator $\Box_k^{\alpha_2}T_m$, we have
\begin{equation*}
  \begin{split}
    \|\Box_k^{\alpha_2}T_mf\|_{M_2}
    \sim &\langle k\rangle^{\frac{-s_2}{1-\alpha_2}}\|T_m(\Box_k^{\alpha_2}f)\|_{M_2^{s_2}}
    \\
    \lesssim &
    \big\|m|~\mathcal{M}_{\mathscr{F}}\left(M_1^{s_1}, M_2^{s_2}\right)\big\| \langle k\rangle^{\frac{-s_2}{1-\alpha_2}}\|\Box_k^{\alpha_2}f\|_{M_1^{s_1}}
    \\
    \lesssim &
    \big\|m|~\mathcal{M}_{\mathscr{F}}\left(M_1^{s_1}, M_2^{s_2}\right)\big\| \langle k\rangle^{\frac{s_1-s_2}{1-\alpha_2}}\|f\|_{M_1},
  \end{split}
\end{equation*}
which implies $\Box_k^{\alpha_2}T_m \in \mathcal {L}(M_1, M_2)$.

Moreover, by the almost disjointization of the decomposition, we choose
a subset of $\mathbb{Z}^n$ denoted by $E_m$, which may depend on the exact $%
m $, such that
\begin{equation}
\Lambda_k^{\alpha_2,\ast}\cap \Lambda_l^{\alpha_2,\ast}=\emptyset
\end{equation}
for any $k,l\in E_m$, $k\neq l$, and
\begin{equation}
\begin{split}
&\big\|m|~W^{\alpha_2}\left(\mathcal{M}_{\mathscr{F}}(M_1, M_2), \mathcal{M}%
_{p}(l_{q_1}^{s_1}, l_{q_2}^{s_2})\right)\big\| \\
\leq& C \big\|\{\|\Box_k^{\alpha_2}T_m\|_{M_1\rightarrow M_2}\}\big\|%
_{l_{q_1}^{s_1,\alpha_2}(E_m)\rightarrow l_{q_2}^{s_2,\alpha_2}(E_m)},
\end{split}%
\end{equation}
where the constant $C$ is independent of the exact $m$.
We assume without loss of generality that $\|\Box_k^{\alpha_2}T_m\|_{M_1\rightarrow M_2}\neq 0$ for all $k\in E_m$.
Then, one
can find $f_k\in \mathscr {S},~f_k\neq 0$ such that
\begin{equation}
\|\Box_k^{\alpha_2}T_mf_k\|_{M_2} \gtrsim
\|\Box_k^{\alpha_2}T_m\|_{M_1\rightarrow M_2} \|f_k\|_{M_1}
\end{equation}
for every $k\in E_m$. It leads to
\begin{equation}
\|\Box_k^{\alpha_2}T_m\Box_k^{\alpha_2,\ast}f_k\|_{M_2} \gtrsim
\|\Box_k^{\alpha_2}T_m\|_{M_1\rightarrow M_2}
\|\Box_k^{\alpha_2,\ast}f_k\|_{M_1}.
\end{equation}
In addition, we deduce $\|\Box_k^{\alpha_2,\ast}f_k\|_{M_1}\neq 0$ by the fact $\|\Box_k^{\alpha_2}T_m\Box_k^{\alpha_2,\ast}f_k\|_{M_2}\neq 0$.

For any nonnegative sequence $\{a_k\}_{k\in \Gamma_m}$, we have that
\begin{equation}
\begin{split}
&\|\{a_k\|\Box_k^{\alpha_2}T_m\|_{M_1\rightarrow M_2}
\|\Box_k^{\alpha_2,\ast}f_k\|_{M_1}\}|~l_{q_2}^{s_2,\alpha_2}(E_m)\| \\
\lesssim&
\|\{a_k\|\Box_k^{\alpha_2}T_m\Box_k^{\alpha_2,\ast}f_k\|_{M_2}%
\}|~l_{q_2}^{s_2,\alpha_2}(E_m)\| \\
\lesssim& \big\|\sum_{k\in E_m}a_k\Box_k^{\alpha_2,\ast}T_mf_k\big\|%
_{M_2^{s_2}} = \big\|T_m\big(\sum_{k\in E_m}a_k\Box_k^{\alpha_2,\ast}f_k\big)%
\big\|_{M_2^{s_2}} \\
\lesssim& \big\|m|~\mathcal{M}_{\mathscr{F}}\left(M_1^{s_1}, M_2^{s_2}\right)%
\big\|\big\|\sum_{k\in E_m}a_k\Box_k^{\alpha_2,\ast}f_k\big\|_{M_1^{s_1}}.
\end{split}
\end{equation}
By the fact that $\left|\{k\in \mathbb{Z}^n:\ \Box_l^{\alpha_1}\Box_k^{\alpha_2, \ast}\neq 0 \}\right|\leq C$ for any $l\in \mathbb{Z}^n$, we deduce
\begin{equation*}
  \begin{split}
    \big\|\sum_{k\in E_m}a_k\Box_k^{\alpha_2,\ast}f_k\big\|_{M_1^{s_1}}
    = &
    \left(\sum_{l\in \mathbb{Z}^n}\langle l\rangle^{\frac{s_1q_1}{1-\alpha_1}}\big\|\sum_{k\in E_m}a_k\Box_l^{\alpha_1}\Box_k^{\alpha_2,\ast}f_k\big\|_{L^{p_1}}^{q_1}\right)^{1/q_1}
    \\
    \lesssim &
    \left(\sum_{l\in \mathbb{Z}^n}\langle l\rangle^{\frac{s_1q_1}{1-\alpha_1}}\sum_{k\in E_m}\big\|a_k\Box_l^{\alpha_1}\Box_k^{\alpha_2,\ast}f_k\big\|_{L^{p_1}}^{q_1}\right)^{1/q_1}
    \\
    = &
    \left(\sum_{k\in E_m}\langle k\rangle^{\frac{s_1q_1}{1-\alpha_2}}a_k^{q_1}\sum_{l\in \mathbb{Z}^n}\big\|\Box_l^{\alpha_1}\Box_k^{\alpha_2,\ast}f_k\big\|_{L^{p_1}}^{q_1}\right)^{1/q_1}
    \\
    \sim &
    \|\{a_k\|\Box_k^{\alpha_2,\ast}f_k\|_{M_1}\}|~l_{q_1}^{s_1,\alpha_2}(E_m)\|.
  \end{split}
\end{equation*}
Thus, we have
\begin{equation}
\begin{split}
&\|\{a_k\|\Box_k^{\alpha_2}T_m\|_{M_1\rightarrow M_2}
\|\Box_k^{\alpha_2,\ast}f_k\|_{M_1}\}|~l_{q_2}^{s_2,\alpha_2}(E_m)\| \\
\lesssim &
\big\|m|~\mathcal{M}_{\mathscr{F}}\left(M_1^{s_1}, M_2^{s_2}\right)\big\|
\|\{a_k\|\Box_k^{\alpha_2,\ast}f_k\|_{M_1}\}|~l_{q_1}^{s_1,%
\alpha_2}(E_m)\|.
\end{split}
\end{equation}

By the arbitrariness of $\{a_k\}_{k\in E_m}$ and the fact $\|\Box_k^{\alpha_2,\ast}f_k\|_{M_1}\neq 0$, we have
\begin{equation}
\big\|\{\|\Box_k^{\alpha_2}T_m\|_{M_1\rightarrow M_2}\}\big\|%
_{l_{q_1}^{s_1,\alpha_2}(E_m)\rightarrow l_{q_2}^{s_2,\alpha_2}(E_m)}
\lesssim \big\|m|~\mathcal{M}_{\mathscr{F}}\left(M_1^{s_1}, M_2^{s_2}\right)%
\big\|.
\end{equation}
So we deduce
\begin{equation}
\big\|m|~W^{\alpha_2}\left(\mathcal{M}_{\mathscr{F}}(M_1, M_2), \mathcal{M}%
_{p}(l_{q_1}^{s_1}, l_{q_2}^{s_2})\right)\big\| \lesssim \big\|m|~\mathcal{M}%
_{\mathscr{F}}\left(M_1^{s_1}, M_2^{s_2}\right)\big\|.
\end{equation}
~\\
\textbf{Case 2. $\alpha_1> \alpha_2$.}

We assume $m\in W^{\alpha_1}\left(\mathcal{M}_{\mathscr{F}}(M_1,
M_2), \mathcal{M}_{p}(l_{q_1}^{s_1}, l_{q_2}^{s_2})\right)$.
By the definition, we obtain
$\Box_k^{\alpha_1}T_m\in \mathcal{L}(M_1, M_2)$ for all
$k\in \mathbb{Z}^n$, and $\{\|\Box_k^{\alpha_1}T_m\|_{M_1\rightarrow
M_2}\}\in \mathcal{M}_{p}(l_{q_1}^{s_1,\alpha_1}, l_{q_2}^{s_2,\alpha_1})$.
For any $f\in \mathscr {S}$, we deduce
\begin{equation}
\begin{split}
\|\Box_k^{\alpha_1}T_m f\|_{M_2}
=&\|\Box_k^{\alpha_1}T_m \Box_k^{\alpha_1,\ast}f\|_{M_2} \\
\lesssim& \|\Box_k^{\alpha_1}T_m\|_{M_1\rightarrow
M_2}\|\Box_k^{\alpha_1,\ast}f\|_{M_1}.
\end{split}%
\end{equation}
It yields
\begin{equation*}
\begin{split}
\|T_mf\|_{M_2^{s_2}}
\sim &
\|\{\|\Box_k^{\alpha_1}T_mf\|_{M_2}\}\|_{l_{q_2}^{s_2,\alpha_1}} \\
\lesssim& \|\{\|\Box_k^{\alpha_1}T_m\|_{M_1\rightarrow
M_2}\|\Box_k^{\alpha_1,\ast}f\|_{M_1}\}\|_{l_{q_2}^{s_2,\alpha_1}} \\
\lesssim& \|\{\|\Box_k^{\alpha_1}T_m\|_{M_1\rightarrow
M_2}\}\|_{l_{q_1}^{s_1,\alpha_1}\rightarrow l_{q_2}^{s_2,\alpha_1}}
\|\{\|\Box_k^{\alpha_1,\ast}f\|_{M_1}\}\|_{l_{q_1}^{s_1,\alpha_1}} \\
\lesssim &
\big\|m|~W^{\alpha_1}\left(\mathcal{M}_{\mathscr{F}}(M_1, M_2),
\mathcal{M}_{p}(l_{q_1}^{s_1}, l_{q_2}^{s_2})\right)\big\| %
\|f\|_{M_1^{s_1}},
\end{split}
\end{equation*}
which further implies
\begin{equation}
\big\|m|~\mathcal{M}_{\mathscr{F}}\left(M_1^{s_1}, M_2^{s_2}\right)\big\| %
\lesssim \big\|m|~W^{\alpha_1}\left(\mathcal{M}_{\mathscr{F}}(M_1, M_2),
\mathcal{M}_{p}(l_{q_1}^{s_1}, l_{q_2}^{s_2})\right)\big\|.
\end{equation}

On the other hand, if $m\in \mathcal{M}_{\mathscr{F}}\left(M_1^{s_1},M_2^{s_2}\right)$.
we deduce $\Box_k^{\alpha_1}T_m \in \mathcal {L}(M_1, M_2)$ as in Case 1.
By the almost disjointization of the decomposition, we choose
a subset of $\mathbb{Z}^n$ denoted by $E_m$, which may depend on the exact $%
m $, such that
\begin{equation}
\Lambda_k^{\alpha_1,\ast}\cap \Lambda_l^{\alpha_1,\ast}=\emptyset
\end{equation}
for any $k,l\in E_m$, $k\neq l$, and
\begin{equation}
\begin{split}
&\big\|m|~W^{\alpha_1}\left(\mathcal{M}_{\mathscr{F}}(M_1, M_2), \mathcal{M}%
_{p}(l_{q_1}^{s_1}, l_{q_2}^{s_2})\right)\big\| \\
\leq& C \big\|\{\|\Box_k^{\alpha_1}T_m\|_{M_1\rightarrow M_2}\}\big\|%
_{l_{q_1}^{s_1,\alpha_1}(E_m)\rightarrow l_{q_2}^{s_2,\alpha_1}(E_m)},
\end{split}%
\end{equation}
where the constant $C$ is independent of the exact $m$.
We assume without loss of generality that $\|\Box_k^{\alpha_1}T_m\|_{M_1\rightarrow M_2}\neq 0$ for all $k\in E_m$.
Then, one
can find $f_k\in \mathscr {S},~f_k\neq 0$ such that
\begin{equation}
\|\Box_k^{\alpha_1}T_mf_k\|_{M_2} \gtrsim
\|\Box_k^{\alpha_1}T_m\|_{M_1\rightarrow M_2} \|f_k\|_{M_1}
\end{equation}
for every $k\in E_m$, which leads to
\begin{equation}
\|\Box_k^{\alpha_1}T_m\Box_k^{\alpha_1,\ast}f_k\|_{M_2} \gtrsim
\|\Box_k^{\alpha_1}T_m\|_{M_1\rightarrow M_2}
\|\Box_k^{\alpha_1,\ast}f_k\|_{M_1}.
\end{equation}
In addition, we deduce $\|\Box_k^{\alpha_1,\ast}f_k\|_{M_1}\neq 0$ by the fact $\|\Box_k^{\alpha_1}T_m\Box_k^{\alpha_1,\ast}f_k\|_{M_2}\neq 0$.

For any nonnegative sequence $\{a_k\}_{k\in \Gamma_m}$, we have that
\begin{equation}
\begin{split}
&\|\{a_k\|\Box_k^{\alpha_1}T_m\|_{M_1\rightarrow M_2}
\|\Box_k^{\alpha_1,\ast}f_k\|_{M_1}\}|~l_{q_2}^{s_2,\alpha_1}(E_m)\| \\
\lesssim&
\|\{a_k\|\Box_k^{\alpha_1}T_m\Box_k^{\alpha_1,\ast}f_k\|_{M_2}\}|~l_{q_2}^{s_2,\alpha_1}(E_m)\|.
\end{split}
\end{equation}
By the spirit of Proposition \ref{proposition, viewpoint}, we deduce
\begin{equation*}
  \begin{split}
    \|\{a_k\|\Box_k^{\alpha_1}T_m\Box_k^{\alpha_1,\ast}f_k\|_{M_2}\}|~l_{q_2}^{s_2,\alpha_1}(E_m)\|
    \lesssim &
    \|\{\|\Box_k^{\alpha_1}\sum_{l\in E_m}a_l\Box_l^{\alpha_1,\ast }T_m f_l\|_{M_2}\}|~l_{q_2}^{s_2,\alpha_1}\|
    \\
    \sim &
    \big\|\sum_{k\in E_m}a_k\Box_k^{\alpha_1,\ast}T_mf_k\big\|_{M_2^{s_2}}
    \\
    = &
    \big\|T_m(\sum_{k\in E_m}a_k\Box_k^{\alpha_1,\ast}f_k)\big\|_{M_2^{s_2}}.
  \end{split}
\end{equation*}
Thus, we obtain
\begin{equation}
\begin{split}
&\|\{a_k\|\Box_k^{\alpha_1}T_m\|_{M_1\rightarrow M_2}
\|\Box_k^{\alpha_1,\ast}f_k\|_{M_1}\}|~l_{q_2}^{s_2,\alpha_1}(E_m)\| \\
\lesssim&
\big\|T_m\big(\sum_{k\in E_m}a_k\Box_k^{\alpha_1,\ast}f_k\big)%
\big\|_{M_2^{s_2}} \\
\lesssim& \big\|m|~\mathcal{M}_{\mathscr{F}}\left(M_1^{s_1}, M_2^{s_2}\right)%
\big\|\big\|\sum_{k\in E_m}a_k\Box_k^{\alpha_1,\ast}f_k\big\|_{M_1^{s_1}}
\\
\lesssim &
\big\|m|~\mathcal{M}_{\mathscr{F}}\left(M_1^{s_1}, M_2^{s_2}\right)\big
\|\{a_k \|\Box_k^{\alpha_1,\ast}f_k\|_{M_1}\}|~l_{q_1}^{s_1,\alpha_1}(E_m)\|.
\end{split}
\end{equation}
By the arbitrariness of $\{a_k\}_{k\in E_m}$ and the fact $\|\Box_k^{\alpha_1,\ast}f_k\|_{M_1}\neq 0$, we have
\begin{equation}
\big\|\{\|\Box_k^{\alpha_1}T_m\|_{M_1\rightarrow M_2}\}\big\|%
_{l_{q_1}^{s_1,\alpha_1}(E_m)\rightarrow l_{q_2}^{s_2,\alpha_1}(E_m)}
\lesssim \big\|m|~\mathcal{M}_{\mathscr{F}}\left(M_1^{s_1}, M_2^{s_2}\right)%
\big\|.
\end{equation}
So we deduce
\begin{equation}
\big\|m|~W^{\alpha_1}\left(\mathcal{M}_{\mathscr{F}}(M_1, M_2), \mathcal{M}%
_{p}(l_{q_1}^{s_1}, l_{q_2}^{s_2})\right)\big\| \lesssim \big\|m|~\mathcal{M}%
_{\mathscr{F}}\left(M_1^{s_1}, M_2^{s_2}\right)\big\|.
\end{equation}
\end{proof}

It is obvious that the embedding relations between $\alpha $-modulation
spaces can be viewed as the boundedness of the identity operator between the
same $\alpha $-modulation spaces. Using Theorem \ref{characterization of
fourier multiplier}, we obtain the reduction of embedding immediately.

\begin{corollary}[\textbf{Reduction of the embedding}]
\label{reduction of the embedding} Let $0< p_i, q_i \leq \infty,$ $s_i\in
\mathbb{R}$, $\alpha_i\in [0,1]$ for $i=1,2$. Then
\begin{equation}
M_{p_1,q_1}^{s_1,\alpha_1}\subseteq M_{p_2,q_2}^{s_2,\alpha_2}
\end{equation}
holds if and only if
\begin{equation}
\left\{\left\|\Box_k^{\alpha_1\vee
\alpha_2}|~M_{p_1,q_1}^{0,\alpha_1}\rightarrow
M_{p_2,q_2}^{0,\alpha_2}\right\|\right\}\in \mathcal{M}_{p}(l_{q_1}^{s_1,%
\alpha_1\vee \alpha_2}, l_{q_2}^{s_2,\alpha_1\vee \alpha_2})
\end{equation}
for $\alpha_1\vee \alpha_2<1$, and
\begin{equation}
\left\{\left\|\Delta_j|~M_{p_1,q_1}^{0,\alpha_1}\rightarrow
M_{p_2,q_2}^{0,\alpha_2}\right\|\right\}\in \mathcal{M}_{p}(l_{q_1}^{s_1,1},
l_{q_2}^{s_2,1})
\end{equation}
for $\alpha_1\vee \alpha_2=1$.
\end{corollary}

\section{Asymptotic estimates for local operators}

In this section, we establish some asymptotic estimates for local operators
between $\alpha$-modulation spaces. For $0<p_1, p_2, q\leqslant\infty$ and $%
(\alpha_1,\alpha_2)\in[0,1]\times[0,1]$, we denote
\begin{equation*}  \label{notation-A}
A (\mathbf{p},q;\alpha_1,\alpha_2)=
\begin{cases}
\begin{aligned} \left[n\alpha_1(\frac{1}{p_1}-\frac{1}{p_2})\right]\vee
\left[n\alpha_2(1-\frac{1}{p_2})-n\alpha_1(1-\frac{1}{p_1})-n(\alpha_2-\alpha_1)\frac{1}{q}\right]
\\ \vee
\left[n(\alpha_2-\alpha_1)(\frac{1}{p_2}-\frac{1}{q})+n\alpha_1(%
\frac{1}{p_1}-\frac{1}{p_2})\right], ~\text{if}~\alpha_1\leq \alpha_2, \\
\left[n\alpha_2(\frac{1}{p_1}-\frac{1}{p_2})\right]\vee
\left[n\alpha_2(1-\frac{1}{p_2})-n\alpha_1(1-\frac{1}{p_1})-n(\alpha_2-%
\alpha_1)\frac{1}{q}\right] \\ \vee
\left[n(\alpha_1-\alpha_2)(\frac{1}{q}-\frac{1}{p_1})+n\alpha_2(%
\frac{1}{p_1}-\frac{1}{p_2})\right], ~\text{if}~\alpha_1> \alpha_2.
\end{aligned}%
\end{cases}%
\end{equation*}

\begin{lemma}[\textbf{Asymptotic estimates}]
\label{asymptotic estimates} Let $0<p_{1}\leq p_{2}\leq \infty $, $0<q\leq
\infty ,$ $s_{i}\in \mathbb{R}$, $\alpha _{i}\in \lbrack 0,1]$ for $i=1,2$.
We have
\begin{equation}
\left\Vert \Box _{k}^{\alpha _{1}\vee \alpha
_{2}}|~M_{p_{1},q}^{0,\alpha _{1}}\rightarrow M_{p_{2},q}^{0,\alpha
_{2}}\right\Vert \sim \langle k\rangle ^{\frac{A(\mathbf{p},q,\alpha
_{1},\alpha _{2})}{1-\alpha _{1}\vee \alpha _{2}}}
\end{equation}%
for $\alpha _{1}\vee \alpha _{2}<1$ and $k\in \mathbb{Z}^{n}$. Also
\begin{equation}
\left\Vert \Delta _{j}|~M_{p_{1},q}^{0,\alpha _{1}}\rightarrow
M_{p_{2},q}^{0,\alpha _{2}}\right\Vert \sim 2^{jA(\mathbf{p},q,\alpha
_{1},\alpha _{2})}
\end{equation}%
for $\alpha _{1}\vee \alpha _{2}=1$ and $j\in \{0\}\cup \mathbb{Z}^{+}$.
\end{lemma}

\begin{proof}
In this proof, we denote $M_i=M_{p_i,q}^{0,\alpha_i}$ for simplicity.
We only state the proof for the case $\alpha_1, \alpha_2<1$, since the proof
of other cases are similar.
\\
\textbf{Case 1. $\alpha_1\leq \alpha_2<1$.}
In this case, we need to show
\begin{equation}
\left\|\Box_k^{\alpha_2}|~M_1\rightarrow M_2\right\| \sim \langle k\rangle^{
\frac{A(\mathbf{p},q,\alpha_1,\alpha_2)}{1-\alpha_2}}\sim 2^{jA(\mathbf{p},q,\alpha_1,\alpha_2)}.
\end{equation}
for each $k\in \mathbb{Z}^n$, $j\in \{0\}\cup \mathbb{Z}^{+}$ and $\langle k\rangle^{\frac{1}{1-\alpha_2}}\sim 2^j$.

Denote
\begin{equation}
  \begin{split}
    A_1(\mathbf{p},q,\alpha_1,\alpha_2)=& n\alpha_1(1/p_1-1/p_2),
    \\
    A_2(\mathbf{p},q,\alpha_1,\alpha_2)=& n\alpha_{2}(1-1/p_{2})-n\alpha _{1}(1-1/p_{1})-n(\alpha _{2}-\alpha _{1})/q,
    \\
    A_3(\mathbf{p},q,\alpha_1,\alpha_2)=& n(\alpha_2-\alpha_1)(1/p_2-1/q)+n\alpha_1(1/p_1-1/p_2).
  \end{split}
\end{equation}
Obviously, we have $A(\mathbf{p},q,\alpha_1,\alpha_2)=\max_{i=1,2,3}A_i(\mathbf{p},q,\alpha_1,\alpha_2)$ for $\alpha_1\leq \alpha_2$.

\textbf{Lower bound estimates.} In this part, we use some special functions
to test the operator norms. Take a smooth function $f$ whose Fourier
transform $\widehat{f}$ has small support near the origin such that $\mathbf{%
supp}\widehat{f_k^{\alpha}}\subset \widetilde{\mathbf{supp}\eta_k^{\alpha}}$
for every $k\in \mathbb{Z}^n$, $\alpha\in [0,1)$, where we denote
\begin{equation}
\widehat{f_k^{\alpha}}=\widehat{f}(\frac{\xi-\langle k\rangle^{\frac{\alpha}{%
1-\alpha}}k}{\langle k\rangle^{\frac{\alpha}{1-\alpha}}}).
\end{equation}

Firstly, we have
\begin{equation}  \label{lower bound 1,1}
\begin{split}
\left\|\Box_k^{\alpha_2}|~M_1\rightarrow M_2\right\| \gtrsim \frac{%
\|\Box_k^{\alpha_2} f_l^{\alpha_1}\|_{M_2}}{\|f_l^{\alpha_1}\|_{M_1}} \sim
\frac{\|f_l^{\alpha_1}\|_{L^{p_2}}}{\|f_l^{\alpha_1}\|_{L^{p_1}}} \sim&
2^{jn\alpha_1(1/p_1-1/p_2)} \\
=& 2^{jA_1(\mathbf{p},q,\alpha_1,\alpha_2)}
\end{split}%
\end{equation}
for some suitable $l\in \mathbb{Z}^n$ such that $\langle l\rangle^{\frac{1}{%
1-\alpha_1}}\sim\langle k\rangle^{\frac{1}{1-\alpha_2}}\sim 2^j$.

Next, a direct calculation yields that
\begin{equation}
\Vert \Box _{k}^{\alpha _{2}}f_{k}^{\alpha _{2}}\Vert _{M_{2}}=\Vert
f_{k}^{\alpha _{2}}\Vert _{M_{2}}\sim \Vert f_{k}^{\alpha _{2}}\Vert
_{L^{p_{2}}}\sim 2^{jn\alpha _{2}(1-1/p_{2})}
\end{equation}%
and
\begin{equation}
\begin{split}
\Vert f_{k}^{\alpha _{2}}\Vert _{M_{1}}=\left( \sum_{l\in \mathbb{Z}%
^{n}}\langle l\rangle ^{\frac{sq}{1-\alpha _{1}}}\Vert \Box _{l}^{\alpha
_{1}}f_{k}^{\alpha _{2}}\Vert _{L^{p}}^{q}\right) ^{1/q}& \lesssim \bigg(%
\sum_{l\in \Gamma _{k}^{\alpha _{1},\alpha _{2}}}\langle l\rangle ^{\frac{sq%
}{1-\alpha }}\Vert \mathscr{F}^{-1}\eta _{l}^{\alpha _{1}}\Vert _{L^{p}}^{q}%
\bigg)^{\frac{1}{q}} \\
& \lesssim 2^{j\alpha _{1}n(1-1/p_{1})}2^{j(\alpha _{2}-\alpha _{1})n/q}.
\end{split}%
\end{equation}%
So we have
\begin{equation}\label{lower bound 1,2}
\begin{split}
\left\Vert \Box _{k}^{\alpha _{2}}|~M_{1}\rightarrow M_{2}\right\Vert
\gtrsim & \frac{\Vert \Box _{k}^{\alpha _{2}}f_{k}^{\alpha _{2}}\Vert
_{M_{2}}}{\Vert f_{k}^{\alpha _{2}}\Vert _{M_{1}}}\gtrsim \frac{2^{jn\alpha
_{2}(1-1/p_{2})}}{2^{j\alpha _{1}n(1-1/p_{1})}2^{j(\alpha _{2}-\alpha
_{1})n/q}} \\
=& 2^{jn\alpha _{2}(1-1/p_{2})}2^{-jn\alpha _{1}(1-1/p_{1})}2^{-jn(\alpha
_{2}-\alpha _{1})/q} \\
=& 2^{jA_{2}(\mathbf{p},q,\alpha _{1},\alpha _{2})}.
\end{split}
\end{equation}

Finally, let
\begin{equation}
F_{k,N}=\sum_{l\in \widetilde{\Gamma _{k}^{\alpha _{1},\alpha _{2}}}%
}T_{Nl}f_{l}^{\alpha _{1}},
\end{equation}%
where $T_{Nl}$ denotes the translation operator: $T_{Nl}f(x)=f(x-Nl)$. By
the almost orthogonality of $\{T_{Nl}f_{l}^{\alpha _{1}}\}_{l\in \widetilde{%
\Gamma _{k}^{\alpha _{1},\alpha _{2}}}}$ as $N\rightarrow \infty $, we
deduce that
\begin{equation}
\Vert \Box _{k}^{\alpha _{2}}F_{k,N}\Vert _{M_{2}}=\Vert F_{k,N}\Vert
_{M_{2}}=\Vert F_{k,N}\Vert _{L^{p_2}}\sim 2^{jn(\alpha _{2}-\alpha
_{1})/p_{2}}2^{jn\alpha _{1}(1-1/p_{2})}
\end{equation}%
as $N\rightarrow \infty $.

On the other hand,
\begin{equation}
\begin{split}
\|F_{k,N}\|_{M_1} &=\left(\sum_{l\in \widetilde{\Gamma_k^{\alpha_1,\alpha_2}}%
}\|f_l^{\alpha_1}\|^{q}_{L^{p_1}}\right)^{\frac{1}{q}} \\
&\sim |\widetilde{\Gamma_k^{\alpha_1,\alpha_2}}|^{1/q}2^{jn\alpha_1(1-1/p_1)}
\\
&\sim 2^{jn(\alpha_2-\alpha_1)/q}2^{jn\alpha_1(1-1/p_1)}.
\end{split}%
\end{equation}
So by the definition of operator norm, we have
\begin{equation}  \label{lower bound 1,3}
\begin{split}
\left\|\Box_k^{\alpha_2}|~M_1\rightarrow M_2\right\| \gtrsim&
\lim_{N\rightarrow \infty}\frac{\|\Box_k^{\alpha_2}F_{k,N}\|_{M_2}}{%
\|F_{k,N}\|_{M_1}} \sim \frac{2^{jn(\alpha_2-\alpha_1)/p_2}2^{jn%
\alpha_1(1-1/p_2)}}{2^{jn(\alpha_2-\alpha_1)/q}2^{jn\alpha_1(1-1/p_1)}} \\
\sim & 2^{jn(\alpha_2-\alpha_1)(1/p_2-1/q)}2^{jn\alpha_1(1/p_1-1/p_2)} \\
=& 2^{jA_3(\mathbf{p},q,\alpha_1,\alpha_2)}.
\end{split}
\end{equation}
Taking together these
estimates, we have the lower bounds
\begin{equation}
\left\|\Box_k^{\alpha_2}|~M_1\rightarrow M_2\right\| \gtrsim \langle
k\rangle^{\frac{A_i(\mathbf{p},q,\alpha_1,\alpha_2)}{1-\alpha}}
\end{equation}
for $i=1,2,3$. Recalling $A(\mathbf{p},q,\alpha_1,\alpha_2)=%
\max_{i=1,2,3}A_i(\mathbf{p},q,\alpha_1,\alpha_2)$, we complete the lower
bound estimates.

\textbf{Upper bound estimates.} Now, we turn to the estimate of upper bound.
Denote
\begin{equation*}
\begin{split}
S=&\{(1/p_1,1/p_2,1/q)\in [0, \infty)^3:~1/p_2\leq 1/p_1\}, \\
S_1=&S\cap \{(1/p_1,1/p_2,1/q):~1/q\geq 1-1/p_2, 1/q\geq 1/p_2\}, \\
S_2=&S\cap \{(1/p_1,1/p_2,1/q):~1/q\leq 1-1/p_2, 1/p_2\leq 1/2\}, \\
S_3=&S\cap \{(1/p_1,1/p_2,1/q):~1/q\leq 1/p_2, 1/p_2\geq 1/2\}.
\end{split}%
\end{equation*}
Obviously, we have
\begin{equation*}
S=S_1\cup S_3 \cup S_3,
\end{equation*}
and
\begin{equation*}
A(\mathbf{p},q,\alpha_1,\alpha_2)=A_i(\mathbf{p},q,\alpha_1,\alpha_2)
\end{equation*}
for $(1/p_1,1/p_2,1/q)\in S_i, \alpha_1\leq \alpha_2$. To verify
$\left\|\Box_k^{\alpha_2}|~M_1\rightarrow M_2\right\| \lesssim \langle
k\rangle^{\frac{A(\mathbf{p},q,\alpha_1,\alpha_2)}{1-\alpha_2}}$, we only need to verify that
\begin{equation*}
  \left\|\Box_k^{\alpha_2}|~M_1\rightarrow M_2\right\| \lesssim \langle
k\rangle^{\frac{A_j(\mathbf{p},q,\alpha_1,\alpha_2)}{1-\alpha_2}}
\end{equation*}
in $S_j$ for $j=1,2,3$.

For $S_1$, we want to verify
\begin{equation}
\left\|\Box_k^{\alpha_2}|~M_1\rightarrow M_2\right\| \lesssim \langle
k\rangle^{\frac{A_1(\mathbf{p},q,\alpha_1,\alpha_2)}{1-\alpha_2}}.
\end{equation}
In fact, once we obtain the estimates for the following 4 cases, the upper
bound in $S_1$ can be deduced by Lemma \ref{positive result for complex
interpolation}.

\textbf{Case 1.1.} $1/p_2=1/q=1/2$, $1/p_2\leq 1/p_1$. We have
\begin{equation}
\|\Box_k^{\alpha_2}f\|_{M_2}=\|\Box_k^{\alpha_2}f\|_{M_{2,2}^{0,\alpha_2}}=\|\Box_k^{\alpha_2}f\|_{M_{2,2}^{0,\alpha_1}}
=\left(\sum_{l\in \Gamma_k^{\alpha_1,\alpha_2}}\|\Box_l^{\alpha_1}\Box_k^{\alpha_2}f\|^2_{L^2}\right)^{1/2}.
\end{equation}
Then we use Lemma \ref{embedding of Lp with Fourier compact support} and Lemma \ref{lemma, convolution} to
deduce that
\begin{equation}
\begin{split}
\|\Box_k^{\alpha_2}f\|_{M_2} \lesssim& \left(\sum_{l\in
\Gamma_k^{\alpha_1,\alpha_2}}\|\Box_l^{\alpha_1}f\|^2_{L^2}\right)^{1/2} \\
\lesssim& 2^{jn\alpha_1(1/p_1-1/p_2)}\left(\sum_{l\in
\Gamma_k^{\alpha_1,\alpha_2}}\|\Box_l^{\alpha_1}f\|^2_{L^{p_1}}\right)^{1/2}
\\
\lesssim& 2^{jn\alpha_1(1/p_1-1/p_2)}\|f\|_{M_1}.
\end{split}%
\end{equation}
Moreover, in this case, we have
\begin{equation}
2^{jn\alpha_1(1/p_1-1/p_2)}=2^{jA_1(\mathbf{p},q,\alpha_1,\alpha_2)}=2^{jA_2(%
\mathbf{p},q,\alpha_1,\alpha_2)}=2^{jA_3(\mathbf{p},q,\alpha_1,\alpha_2)}.
\end{equation}

\textbf{Case 1.2.} $1/p_2=1/p_1=0$, $1/q\geq 1$. We have
\begin{equation}
\begin{split}
\|\Box_k^{\alpha_2}f\|_{M_2} =& \|\Box_k^{\alpha_2}f\|_{L^{\infty}}
\lesssim  \sum_{l\in
\Gamma_k^{\alpha_1,\alpha_2}}\|\Box_l^{\alpha_1}f\|_{L^{\infty}}
\\
\lesssim &
\left(\sum_{l\in
\Gamma_k^{\alpha_1,\alpha_2}}\|\Box_l^{\alpha_1}f\|^q_{L^{\infty}}%
\right)^{1/q} \lesssim \|f\|_{M_1}=2^{jA_1(\mathbf{p},q,\alpha_1,\alpha_2)}\|f\|_{M_1}.
\end{split}%
\end{equation}

\textbf{Case 1.3.} $1/p_2=1/p_1=1/q\geq 1$. We use Lemma \ref{convolution} to
deduce that
\begin{equation}
\begin{split}
\|\Box_k^{\alpha_2}f\|_{M_2} =& \|\Box_k^{\alpha_2}f\|_{L^{p_2}} =
\|\sum_{l\in
\Gamma_k^{\alpha_1,\alpha_2}}\Box_l^{\alpha_1}\Box_k^{\alpha_2}f\|_{L^{p_2}}
\\
\lesssim & \left(\sum_{l\in
\Gamma_k^{\alpha_1,\alpha_2}}\|\Box_l^{\alpha_1}f\|^{p_2}_{L^{p_2}}%
\right)^{1/p_2} \lesssim \|f\|_{M_1}.
\end{split}%
\end{equation}
Moreover, in this case, we have
\begin{equation}
1=2^{jA_1(\mathbf{p},q,\alpha_1,\alpha_2)}=2^{jA_3(\mathbf{p}%
,q,\alpha_1,\alpha_2)}.
\end{equation}

\textbf{Case 1.4.} $1/p_2=0$, $1/q=1.$ We have that
\begin{equation}
\begin{split}
\|\Box_k^{\alpha_2}f\|_{M_2} =& \|\Box_k^{\alpha_2}f\|_{L^{\infty}} \lesssim
\sum_{l\in \Gamma_k^{\alpha_1,\alpha_2}}\|\Box_l^{\alpha_1}f\|_{L^{\infty}}
\\
\lesssim & 2^{jn\alpha_1(1/p_1-1/p_2)}\sum_{l\in
\Gamma_k^{\alpha_1,\alpha_2}}\|\Box_l^{\alpha_1}f\|_{L^{p_1}} \\
\lesssim & 2^{jn\alpha_1(1/p_1-1/p_2)}\|f\|_{M_1}.
\end{split}%
\end{equation}
Moreover, in this case, we have
\begin{equation}
2^{jn\alpha_1(1/p_1-1/p_2)}=2^{jA_1(\mathbf{p},q,\alpha_1,\alpha_2)}=2^{jA_2(%
\mathbf{p},q,\alpha_1,\alpha_2)}.
\end{equation}

Combining with the estimates of Case 1.1, Case 1.2, Case 1.3, Case 1.4, we use the
interpolation theory to obtain the upper bound estimates for $S_1$.

For $S_{2}$, we want to verify
\begin{equation}
\left\Vert \Box _{k}^{\alpha _{2}}|~M_{1}\rightarrow M_{2}\right\Vert
\lesssim \langle k\rangle ^{\frac{A_{2}(\mathbf{p},q,\alpha _{1},\alpha _{2})%
}{1-\alpha _{2}}}.
\end{equation}%
To this end, we need the estimates in the following Case 1.5 and Case 1.6.

\textbf{Case 1.5.} $1/2=1/p_2\leq1/p_1$, $1/q=0$. We have that
\begin{equation}
\|\Box_k^{\alpha_2}f\|_{M_2}=\|\Box_k^{\alpha_2}f\|_{L^2}=\|\Box_k^{%
\alpha_2}f\|_{M_{2,2}^{0,\alpha_1}}.
\end{equation}
It then yields
\begin{equation}
\begin{split}
\|\Box_k^{\alpha_2}f\|_{M_2} \lesssim& \left(\sum_{l\in
\Gamma_k^{\alpha_1,\alpha_2}}\|\Box_l^{\alpha_1}f\|^2_{L^2}\right)^{1/2} \\
\lesssim& |\Gamma_k^{\alpha_1,\alpha_2}|^{1/2}\sup_{l\in
\Gamma_k^{\alpha_1,\alpha_2}}\|\Box_l^{\alpha_1}f\|_{L^{2}} \\
\lesssim&
|\Gamma_k^{\alpha_1,\alpha_2}|^{1/2}2^{jn\alpha_1(1/p_1-1/p_2)}\sup_{l\in
\Gamma_k^{\alpha_1,\alpha_2}}\|\Box_l^{\alpha_1}f\|_{L^{p_1}} \\
\lesssim&
2^{jn(\alpha_2-\alpha_1)/2}2^{jn\alpha_1(1/p_1-1/p_2)}\|f\|_{M_{p_1,q}^{0,%
\alpha_1}}\sim2^{jA_2(\mathbf{p},q,\alpha_1,\alpha_2)}\|f\|_{M_1}.
\end{split}%
\end{equation}
Moreover, in this case, we have
\begin{equation}
2^{jA_2(\mathbf{p},q,\alpha_1,\alpha_2)}=2^{jA_3(\mathbf{p}%
,q,\alpha_1,\alpha_2)}.
\end{equation}

\textbf{Case 1.6.} $1/p_2=1/q=0$. We have
\begin{equation}
\|\Box_k^{\alpha_2}f\|_{M_2} = \|\Box_k^{\alpha_2}f\|_{L^{\infty}} \\
\lesssim \sum_{l\in
\Gamma_k^{\alpha_1,\alpha_2}}\|\Box_l^{\alpha_1}f\|_{L^{\infty}}.
\end{equation}
It leads to
\begin{equation}
\begin{split}
\|\Box_k^{\alpha_2}f\|_{M_2} \lesssim & \sum_{l\in
\Gamma_k^{\alpha_1,\alpha_2}}\|\Box_l^{\alpha_1}f\|_{L^{\infty}} \\
\lesssim& |\Gamma_k^{\alpha_1,\alpha_2}|\sup_{l\in
\Gamma_k^{\alpha_1,\alpha_2}}\|\Box_l^{\alpha_1}f\|_{L^{\infty}} \\
\lesssim&
|\Gamma_k^{\alpha_1,\alpha_2}|2^{jn\alpha_1(1/p_1-1/p_2)}\sup_{l\in
\Gamma_k^{\alpha_1,\alpha_2}}\|\Box_l^{\alpha_1}f\|_{L^{p_1}} \\
\lesssim&
2^{jn(\alpha_2-\alpha_1)}2^{jn\alpha_1(1/p_1-1/p_2)}\|f\|_{M_{p_1,q}^{0,%
\alpha_1}}\sim2^{jA_2(\mathbf{p},q,\alpha_1,\alpha_2)}\|f\|_{M_1}.
\end{split}%
\end{equation}
Combining with the estimates of Case 1.1, Case 1.4, Case 1.5, Case 1.6, we use the
interpolation theory to deduce the upper bound estimates for $S_2$.

For $S_3$, we want to verify
\begin{equation}
\left\|\Box_k^{\alpha_2}|~M_1\rightarrow M_2\right\| \lesssim \langle
k\rangle^{\frac{A_3(\mathbf{p},q,\alpha_1,\alpha_2)}{1-\alpha_2}}.
\end{equation}
We further need the estimate in following.

\textbf{Case 1.7.} $1/p_2=1/p_1\geq 1$, $1/q=0$. In this case, we have that
\begin{equation}
\begin{split}
\|\Box_k^{\alpha_2}f\|_{M_2} =& \|\Box_k^{\alpha_2}f\|_{L^{p_2}} =
\|\sum_{l\in
\Gamma_k^{\alpha_1,\alpha_2}}\Box_l^{\alpha_1}\Box_k^{\alpha_2}f\|_{L^{p_2}}
\\
\lesssim & \left(\sum_{l\in
\Gamma_k^{\alpha_1,\alpha_2}}\|\Box_l^{\alpha_1}f\|^{p_2}_{L^{p_2}}%
\right)^{1/p_2} \\
\lesssim & |\Gamma_k^{\alpha_1,\alpha_2}|^{1/p_2}\sup_{l\in
\Gamma_k^{\alpha_1,\alpha_2}}\|\Box_l^{\alpha_1}f\|_{L^{p_2}} \\
\lesssim & 2^{jn(\alpha_2-\alpha_1)/p_2}\|f\|_{M_1} \sim 2^{jA_3(\mathbf{p}%
,q,\alpha_1,\alpha_2)}\|f\|_{M_1}.
\end{split}%
\end{equation}
Combining with the estimates in Case 1.1, Case 1.3, Case 1.5, Case 1.7, we use the
interpolation theory to deduce the upper bound estimates for $S_3$.
\\
\textbf{Case 2. $\alpha_2< \alpha_1<1$.}
In this case, we need to show
\begin{equation}
\left\|\Box_k^{\alpha_1}|~M_1\rightarrow M_2\right\| \sim \langle k\rangle^{
\frac{A(\mathbf{p},q,\alpha_1,\alpha_2)}{1-\alpha_1}}\sim 2^{jA(\mathbf{p},q,\alpha_1,\alpha_2)}.
\end{equation}
for each $k\in \mathbb{Z}^n$, $j\in \{0\}\cup \mathbb{Z}^{+}$ and $\langle k\rangle^{\frac{1}{1-\alpha_1}}\sim 2^j$.

Denote
\begin{equation}
  \begin{split}
    \widetilde{A_1}(\mathbf{p},q,\alpha_1,\alpha_2)=& n\alpha_2(1/p_1-1/p_2)
    \\
    \widetilde{A_2}(\mathbf{p},q,\alpha_1,\alpha_2)=& n\alpha_{2}(1-1/p_{2})-n\alpha _{1}(1-1/p_{1})-n(\alpha _{2}-\alpha _{1})/q
    \\
    \widetilde{A_3}(\mathbf{p},q,\alpha_1,\alpha_2)=& n(\alpha_1-\alpha_2)(1/q-1/p_1)+n\alpha_2(1/p_1-1/p_2).
  \end{split}
\end{equation}
We have $A(\mathbf{p},q,\alpha_1,\alpha_2)=\max_{i=1,2,3}\widetilde{A_i}(\mathbf{p},q,\alpha_1,\alpha_2)$ for $\alpha_1>\alpha_2$.

\textbf{Lower bound estimates.} As in the Case 1, we take a smooth function $f$ whose Fourier
transform $\widehat{f}$ has small support near the origin such that $\mathbf{%
supp}\widehat{f_k^{\alpha}}\subset \widetilde{\mathbf{supp}\eta_k^{\alpha}}$
for every $k\in \mathbb{Z}^n$, $\alpha\in [0,1)$.

By the same method in Case 1, we have the following estimates.
Firstly, we have
\begin{equation}  \label{lower bound 2,1}
\begin{split}
\left\|\Box_k^{\alpha_1}|~M_1\rightarrow M_2\right\| \gtrsim \frac{
\|\Box_k^{\alpha_1} f_l^{\alpha_2}\|_{M_2}}{\|f_l^{\alpha_2}\|_{M_1}}
\sim
2^{j\widetilde{A_1}(\mathbf{p},q,\alpha_1,\alpha_2)}
\end{split}%
\end{equation}
for some suitable $l\in \mathbb{Z}^n$ such that $\langle l\rangle^{\frac{1}{%
1-\alpha_1}}\sim\langle k\rangle^{\frac{1}{1-\alpha_2}}\sim 2^j$.

We also deduce
\begin{equation}\label{lower bound 2,2}
\begin{split}
\left\Vert \Box _{k}^{\alpha _{1}}|~M_{1}\rightarrow M_{2}\right\Vert
\gtrsim & \frac{\Vert \Box _{k}^{\alpha _{1}}f_{k}^{\alpha _{1}}\Vert
_{M_{2}}}{\Vert f_{k}^{\alpha _{1}}\Vert _{M_{1}}}
\gtrsim
2^{j\widetilde{A_{2}}(\mathbf{p},q,\alpha _{1},\alpha _{2})}.
\end{split}
\end{equation}

Finally, take
\begin{equation}
G_{k,N}=\sum_{l\in \widetilde{\Gamma _{k}^{\alpha _{2},\alpha _{1}}}%
}T_{Nl}f_{l}^{\alpha _{2}}
\end{equation}
We deduce
\begin{equation}  \label{lower bound 2,3}
\begin{split}
\left\|\Box_k^{\alpha_1}|~M_1\rightarrow M_2\right\|
\gtrsim
\lim_{N\rightarrow \infty}\frac{\|\Box_k^{\alpha_1}G_{k,N}\|_{M_2}}{%
\|G_{k,N}\|_{M_1}}
\sim
2^{j\widetilde{A_3}(\mathbf{p},q,\alpha_1,\alpha_2)}.
\end{split}
\end{equation}
Taking together these
estimates, we have the lower bounds
\begin{equation}
\left\|\Box_k^{\alpha_1}|~M_1\rightarrow M_2\right\| \gtrsim \langle
k\rangle^{\frac{\widetilde{A_i}(\mathbf{p},q,\alpha_1,\alpha_2)}{1-\alpha}}
\end{equation}
for $i=1,2,3$. Recalling $A(\mathbf{p},q,\alpha_1,\alpha_2)=%
\max_{i=1,2,3}\widetilde{A_i}(\mathbf{p},q,\alpha_1,\alpha_2)$, we complete the lower bound estimates for this case.

\textbf{Upper bound estimates.} We turn to the estimate of upper bound in this case.
We divide the area
\begin{equation*}
S=\{(1/p_1,1/p_2,1/q)\in [0, \infty)^3:~1/p_2\leq 1/p_1\}
\end{equation*}
into 3 zones as following.
\begin{equation*}
\begin{split}
\widetilde{S_1}=&S\cap \{(1/p_1,1/p_2,1/q):~1/q\leq 1-1/p_1, 1/q\leq 1/p_1\}, \\
\widetilde{S_2}=&S\cap \{(1/p_1,1/p_2,1/q):~1/q\geq 1-1/p_1, 1/p_1\geq 1/2\}, \\
\widetilde{S_3}=&S\cap \{(1/p_1,1/p_2,1/q):~1/q\geq 1/p_1, 1/p_1\leq 1/2\}.
\end{split}%
\end{equation*}
One can easily verify that
\begin{equation*}
A(\mathbf{p},q,\alpha_1,\alpha_2)=\widetilde{A_i}(\mathbf{p},q,\alpha_1,\alpha_2)
\end{equation*}
for $(1/p_1,1/p_2,1/q)\in \widetilde{S_i}.$

For $\widetilde{S_1}$, we want to verify
\begin{equation}
\left\|\Box_k^{\alpha_1}|~M_1\rightarrow M_2\right\| \lesssim \langle
k\rangle^{\frac{\widetilde{A_1}(\mathbf{p},q,\alpha_1,\alpha_2)}{1-\alpha_1}}.
\end{equation}
We deduce the estimates for the following 4 cases, then the upper
bound in $\widetilde{S_1}$ can be deduced by Lemma \ref{positive result for complex
interpolation}.

\textbf{Case 2.1.} $1/p_1=1/q=1/2$, $1/p_2\leq 1/p_1$.
Using Lemma \ref{embedding of Lp with Fourier compact support}, We deduce
\begin{equation}
\begin{split}
  \|\Box_k^{\alpha_1}f\|_{M_2}
  = &\left(\sum_{l\in \Gamma_k^{\alpha_2,\alpha_1}} \|\Box_l^{\alpha_2}\Box_k^{\alpha_1}f\|^2_{L^{p_2}}\right)^{1/2}
  \\
  \lesssim & 2^{j\alpha_2n(1/p_1-1/p_2)}\left(\sum_{l\in \Gamma_k^{\alpha_2,\alpha_1}} \|\Box_l^{\alpha_2}\Box_k^{\alpha_1}f\|^2_{L^{2}}\right)^{1/2}
  \\
  = & 2^{j\alpha_2n(1/p_1-1/p_2)}\|\Box_k^{\alpha_1}f\|_{M_{2,2}^{\alpha_2}}
  \\
  \sim & 2^{j\alpha_2n(1/p_1-1/p_2)}\|\Box_k^{\alpha_1}f\|_{M_{2,2}^{\alpha_1}}
  \lesssim 2^{j\alpha_2n(1/p_1-1/p_2)}\|f\|_{M_1}.
\end{split}
\end{equation}
Moreover, in this case, we have
\begin{equation}
2^{jn\alpha_2(1/p_1-1/p_2)}=2^{j\widetilde{A_1}(\mathbf{p},q,\alpha_1,\alpha_2)}
=2^{j\widetilde{A_2}(\mathbf{p},q,\alpha_1,\alpha_2)}
=2^{j\widetilde{A_3}(\mathbf{p},q,\alpha_1,\alpha_2)}.
\end{equation}

\textbf{Case 2.2.} $1/p_2=1/p_1=0$. We obtain
\begin{equation}
  \begin{split}
    \|\Box_k^{\alpha_1}f\|_{M_2}
    = &
    \left(\sum_{l\in \Gamma_k^{\alpha_2,\alpha_1}}\|\Box_l^{\alpha_2}\Box_k^{\alpha_1}f\|^q_{L^{\infty}}\right)^{1/q}
    \\
    \lesssim &
    |\Gamma_k^{\alpha_2,\alpha_1}|^{1/q}\|\Box_k^{\alpha_1}f\|_{L^{\infty}}
    \\
    \lesssim &
    2^{j(\alpha_1-\alpha_2)n/q}\|f\|_{M_1}\sim 2^{j\widetilde{A_3}(\mathbf{p},q,\alpha_1,\alpha_2)}\|f\|_{M_1}.
  \end{split}
\end{equation}
We also notice that $\widetilde{A_3}(\mathbf{p},q,\alpha_1,\alpha_2)=\widetilde{A_1}(\mathbf{p},q,\alpha_1,\alpha_2)$ at $1/p_1=1/p_2=1/q=0$.

\textbf{Case 2.3.} $1/p_2=1/p_1\geq 1$, $1/q=0$.  We have
\begin{equation}
  \|\Box_k^{\alpha_1}\|_{M_2}=\sup_{l\in \Gamma_k^{\alpha_2,\alpha_1}}\|\Box_l^{\alpha_2}\Box_k^{\alpha_1}f\|_{L^{p_2}}.
\end{equation}
Using Lemma \ref{lemma, convolution}, we obtain
\begin{equation}
  \begin{split}
  \|\Box_l^{\alpha_2}\Box_k^{\alpha_1}f\|_{L^{p_2}}
  \lesssim &\langle k\rangle^{\frac{\alpha_1n(1/p_2-1)}{1-\alpha_1}} \|\mathscr{F}^{-1}\eta_l^{\alpha_2}\|_{L^{p_2}}\|\Box_k^{\alpha_1}f\|_{L^{p_2}}
  \\
  \sim &
  2^{jn(\alpha_2-\alpha_1)(1-1/p_2)}\|\Box_k^{\alpha_1}f\|_{L^{p_2}}.
  \end{split}
\end{equation}
Thus
\begin{equation}
  \begin{split}
  \|\Box_k^{\alpha_1}\|_{M_2}
  \lesssim &
  2^{jn(\alpha_2-\alpha_1)(1-1/p_2)}\|\Box_k^{\alpha_1}f\|_{L^{p_1}}
  \\
  \lesssim &
  2^{jn(\alpha_2-\alpha_1)(1-1/p_2)}\|f\|_{M_1}=2^{j\widetilde{A_2}(\mathbf{p},q,\alpha_1,\alpha_2)}\|f\|_{M_1}.
  \end{split}
\end{equation}
We also notice that $\widetilde{A_2}(\mathbf{p},q,\alpha_1,\alpha_2)=\widetilde{A_1}(\mathbf{p},q,\alpha_1,\alpha_2)$ at $1/p_2=1/p_1= 1$, $1/q=0$.

\textbf{Case 2.4.} $1/p_2=1/q=0$, $1/p_1\geq 1$. In this case, we have
\begin{equation}
  \begin{split}
  \|\Box_k^{\alpha_1}\|_{M_2}
  = &
  \sup_{l\in \Gamma_k^{\alpha_2,\alpha_1}}\|\Box_l^{\alpha_2}\Box_k^{\alpha_1}f\|_{L^{\infty}}
  \\
  \lesssim &
  2^{jn\alpha_2/p_1}\sup_{l\in \Gamma_k^{\alpha_2,\alpha_1}}\|\Box_l^{\alpha_2}\Box_k^{\alpha_1}f\|_{L^{p_1}}
  \\
  \lesssim &
  2^{jn\alpha_2/p_1}2^{jn(\alpha_2-\alpha_1)(1-1/p_1)}\|\Box_k^{\alpha_1}f\|_{L^{p_1}}
    \\
  \sim &
  2^{j\widetilde{A_2}(\mathbf{p},q,\alpha_1,\alpha_2)}\|\Box_k^{\alpha_1}f\|_{L^{p_1}}
  \lesssim
  2^{j\widetilde{A_2}(\mathbf{p},q,\alpha_1,\alpha_2)}\|f\|_{M_1}.
  \end{split}
\end{equation}
We also notice that $\widetilde{A_2}(\mathbf{p},q,\alpha_1,\alpha_2)=\widetilde{A_1}(\mathbf{p},q,\alpha_1,\alpha_2)$
at $1/p_2=1/q=0$, $1/p_1= 1$.

Combining with the estimates of Case 2.1, Case 2.2, Case 2.3, Case 2.4, we use the
interpolation theory to obtain the upper bound estimates for $\widetilde{S_1}$.

For $\widetilde{S_{2}}$, we want to verify
\begin{equation}
\left\Vert \Box _{k}^{\alpha _{1}}|~M_{1}\rightarrow M_{2}\right\Vert
\lesssim \langle k\rangle ^{\frac{\widetilde{A_{2}}(\mathbf{p},q,\alpha _{1},\alpha _{2})}{1-\alpha _{1}}}.
\end{equation}
To this end, we need the estimates in the following Case 2.5 and Case 2.6.

\textbf{Case 2.5.} $1/p_2=1/p_1=1/2$, $1/q\geq 1/2$. By the H\"{o}lder's inequality, we deduce
\begin{equation}
  \begin{split}
    \|\Box_k^{\alpha_1}\|_{M_2}
    = &
    \left(\sum_{l\in \Gamma_k^{\alpha_2,\alpha_1}}\|\Box_l^{\alpha_2}\Box_k^{\alpha_1}f\|^q_{L^2}\right)^{1/q}
    \\
    \lesssim &
    |\Gamma_k^{\alpha_2,\alpha_1}|^{1/q-1/2}\left(\sum_{l\in \Gamma_k^{\alpha_2,\alpha_1}}\|\Box_l^{\alpha_2}\Box_k^{\alpha_1}f\|^2_{L^2}\right)^{1/2}
    \\
    \sim &
    2^{jn(\alpha_1-\alpha_2)(1/q-1/2)}\|\Box_k^{\alpha_1}f\|_{L^2}
    \\
    \lesssim &
    2^{jn(\alpha_1-\alpha_2)(1/q-1/2)}\|f\|_{M_1}.
  \end{split}
\end{equation}
Moreover, we have
\begin{equation}
  2^{jn(\alpha_1-\alpha_2)(1/q-1/2)}=2^{j\widetilde{A_2}(\mathbf{p},q,\alpha_1,\alpha_2)}=2^{j\widetilde{A_3}(\mathbf{p},q,\alpha_1,\alpha_2)}
\end{equation}
in this case.

\textbf{Case 2.6.} $1/p_2=0$, $1/p_1=1/2$, $1/q\geq 1/2$.
Using Lemma \ref{embedding of Lp with Fourier compact support}
\begin{equation}
  \begin{split}
    \|\Box_k^{\alpha_1}f\|_{M_2}
    = &
    \left(\sum_{l\in \Gamma_k^{\alpha_2,\alpha_1}}\|\Box_l^{\alpha_2}\Box_k^{\alpha_1}f\|^q_{L^{\infty}}\right)^{1/q}
    \\
    \lesssim &
    2^{jn\alpha_2/2}\left(\sum_{l\in \Gamma_k^{\alpha_2,\alpha_1}}\|\Box_l^{\alpha_2}\Box_k^{\alpha_1}f\|^q_{L^{2}}\right)^{1/q}.
  \end{split}
\end{equation}
Then, we use the conclusion of Case 2.5 to deduce
\begin{equation}
  \begin{split}
    2^{jn\alpha_2/2}\left(\sum_{l\in \Gamma_k^{\alpha_2,\alpha_1}}\|\Box_l^{\alpha_2}\Box_k^{\alpha_1}f\|^q_{L^{2}}\right)^{1/q}
    \lesssim
    2^{jn\alpha_2/2}2^{jn(\alpha_1-\alpha_2)(1/q-1/2)}\|f\|_{M_1}.
  \end{split}
\end{equation}
It follows that
\begin{equation}
  \begin{split}
    \|\Box_k^{\alpha_1}f\|_{M_2}
    \lesssim
    2^{jn\alpha_2/2}2^{jn(\alpha_1-\alpha_2)(1/q-1/2)}\|f\|_{M_1}.
  \end{split}
\end{equation}
Moreover, we have
\begin{equation}
  2^{jn\alpha_2/2}2^{jn(\alpha_1-\alpha_2)(1/q-1/2)}=2^{j\widetilde{A_2}(\mathbf{p},q,\alpha_1,\alpha_2)}=2^{j\widetilde{A_3}(\mathbf{p},q,\alpha_1,\alpha_2)}
\end{equation}
in this case.

Combining with the estimates of Case 2.3, Case 2.4, Case 2.5 and Case 2.6, we obtain the upper bound estimates for $\widetilde{S_2}$.
In addition, we use the estimates of Case 2.2, Case 2.5 and Case 2.6 to deduce
\begin{equation}
\left\Vert \Box _{k}^{\alpha _{1}}|~M_{1}\rightarrow M_{2}\right\Vert
\lesssim \langle k\rangle ^{\frac{\widetilde{A_{3}}(\mathbf{p},q,\alpha _{1},\alpha _{2})}{1-\alpha _{1}}}.
\end{equation}
for $\widetilde{S_3}$.

\end{proof}

\section{The embedding  relations between $\protect\alpha$-modulation spaces}

For $0<p_1, p_2, q\leqslant\infty$ and $(\alpha_1,\alpha_2)\in[0,1]\times[0,1%
]$, we recall that
\begin{equation*}
R (\mathbf{p},\mathbf{q};\alpha_1,\alpha_2)=
\begin{cases}
A (\mathbf{p},q_1;\alpha_1,\alpha_2),~ \text{if}~ \alpha_1\leq \alpha_2, \\
A (\mathbf{p},q_2;\alpha_1,\alpha_2),~ \text{if}~ \alpha_1\geq \alpha_2.%
\end{cases}%
\end{equation*}
After the preparation in the last two sections, we are now in a position to
give the final proof of Theorem \ref{sharpness of embedding} in the
following. \newline
\textbf{Proof of Theorem \ref{sharpness of embedding}.} We only show the
proof for $\alpha_1,\alpha_2<1$, the other cases can be treated similarly.

Firstly, we claim that $1/p_2 \leq 1/p_1$ is necessary if the embedding
relation holds. In fact, we can choose a smooth function $h$ whose Fourier
transform $\widehat{h}$ has small compact support, and denote $\widehat{%
h_{\lambda}}(\xi)=\widehat{h}(\frac{\xi}{\lambda}) .$ Then the embedding $%
M_{p_1,q_1}^{s_1,\alpha_1}\subseteq M_{p_2,q_2}^{s_2,\alpha_2}$ implies
\begin{equation}
\|h_{\lambda}\|_{L^{p_2}}\lesssim \|h_{\lambda}\|_{L^{p_1}}
\end{equation}
as $\lambda \rightarrow 0$, which implies $1/p_2 \leq 1/p_1$.

On the other hand, one can easily verify that
\begin{equation*}
\left\|\Box_k^{\alpha_2}|~M_{p_1,q_1}^{0,\alpha_1}\rightarrow
M_{p_2,q_2}^{0,\alpha_2}\right\|\sim
\left\|\Box_k^{\alpha_2}|~M_{p_1,q_1}^{0,\alpha_1}\rightarrow
M_{p_2,q_1}^{0,\alpha_2}\right\|.
\end{equation*}
for $\alpha_1\leq \alpha_2$, and
\begin{equation*}
\left\|\Box_k^{\alpha_1}|~M_{p_1,q_1}^{0,\alpha_1}\rightarrow
M_{p_2,q_2}^{0,\alpha_2}\right\|\sim
\left\|\Box_k^{\alpha_1}|~M_{p_1,q_2}^{0,\alpha_1}\rightarrow
M_{p_2,q_2}^{0,\alpha_2}\right\|.
\end{equation*}
for $\alpha_1\geq \alpha_2$. So Lemma \ref{asymptotic estimates} implies
\begin{equation}
\left\|\Box_k^{\alpha_1\vee \alpha_2}|~M_{p_1,q_1}^{0,\alpha_1}\rightarrow
M_{p_2,q_2}^{0,\alpha_2}\right\| \sim \langle k\rangle^{\frac{R(\mathbf{p},%
\mathbf{q},\alpha_1,\alpha_2)}{1-\alpha_1\vee \alpha_2}}.
\end{equation}
Now, we divide the relations between $1/q_2$ and $1/q_1$ into two cases.%
\newline
\textbf{Case One:~$1/q_2\leq 1/q_1$.} \newline
In this case, we have $\mathcal{M}_{p}(l_{q_1}^{s_1,\alpha_1\vee \alpha_2},
l_{q_2}^{s_2,\alpha_1\vee \alpha_2})=l_{\infty}^{s_2-s_1,\alpha_1\vee
\alpha_2}$ and
\begin{equation}  \label{for proof 2}
\big\|\{\langle k\rangle^{\frac{R(\mathbf{p},\mathbf{q},\alpha_1,\alpha_2)}{%
1-\alpha_1\vee \alpha_2}}\}\big\|_{l_{\infty}^{s_2-s_1,\alpha_1\vee
\alpha_2}} =\sup_{k\in \mathbb{Z}^n}\langle k\rangle^{\frac{s_2-s_1}{%
1-\alpha_1\vee \alpha_2}}\langle k\rangle^{\frac{R(\mathbf{p},\mathbf{q}%
,\alpha_1,\alpha_2)}{1-\alpha_1\vee \alpha_2}}.
\end{equation}
Thus, it is obvious that $s_2+R(\mathbf{p},\mathbf{q},\alpha_1,\alpha_2)\leq
s_1$ is the sufficient and necessary condition for the boundedness of (\ref%
{for proof 2}). \newline
\textbf{Case Two:~$1/q_2> 1/q_1$.} \newline
In this case, we have $\mathcal{M}_{p}(l_{q_1}^{s_1,\alpha_1\vee \alpha_2},
l_{q_2}^{s_2,\alpha_1\vee \alpha_2})=l_{r}^{s_2-s_1,\alpha_1\vee \alpha_2}$,
where $1/r=1/q_2-1/q_1.$
\begin{equation}  \label{for proof 3}
\big\|\{\langle k\rangle^{\frac{R(\mathbf{p},\mathbf{q},\alpha_1,\alpha_2)}{%
1-\alpha_1\vee \alpha_2}}\}\big\|_{l_{r}^{s_2-s_1,\alpha_1\vee \alpha_2}}
=\left(\sum_{k\in \mathbb{Z}^n}\langle k\rangle^{r\big[\frac{s_2-s_1}{%
1-\alpha_1\vee \alpha_2}+\frac{R(\mathbf{p},\mathbf{q},\alpha_1,\alpha_2)}{%
1-\alpha_1\vee \alpha_2}\big]}\right)^{1/r}.
\end{equation}
One can easily verify that $s_2+R(\mathbf{p},\mathbf{q},\alpha_1,\alpha_2)+%
\frac{n(1-\alpha_1\vee\alpha_2)}{q_2}< s_1+\frac{n(1-\alpha_1\vee\alpha_2)}{%
q_1}$ is the sharp condition for the boundedness of (\ref{for proof 3}).

We now complete the proof of Theorem \ref{sharpness of embedding} with the
aid of Corollary \ref{reduction of the embedding}.

\noindent \textbf{ Acknowledgements.}\quad This work is partially supported by the NNSF of China (Grant Nos. 11201103, 11371295, 11471041).

\end{document}